\documentclass[a4paper,12pt]{amsart}

\title{An ancient Ricci flow emerging from Taub-Bolt}
\author{John Hughes}
\usepackage{mathtools}
\usepackage{amsmath}
\usepackage{amssymb}
\usepackage{amsfonts}
\usepackage{amsthm}

\newtheorem{lemma}{Lemma}[section]
\newtheorem{theorem}{Theorem}[section]

\newtheorem{proposition}{Proposition}[section]

\theoremstyle{definition}
\newtheorem{definition}[theorem]{Definition}
\newtheorem{remark}[theorem]{Remark}

\usepackage{graphicx}
\graphicspath{ {./images/} }
\usepackage{amsmath, amssymb, amsfonts, graphicx, pgfplots, tikz, tikz-cd, stmaryrd}
\newtheorem*{question*}{Question}
\date{}
\usepackage{enumerate}
\usepackage{hyperref}
\usepackage{fullpage}
\usepackage{tensor}
\usepackage{bigints}
\address{Mathematical Institute, University of Oxford, Oxford OX2 6GG, United Kingdom}
\email{john.hughes@maths.ox.ac.uk}
\begin{document}
	
\maketitle 
	\begin{abstract}
This paper proves that there exists a non-trivial ancient solution to the Ricci flow emerging from the Taub-Bolt metric.
\end{abstract}

\section{Introduction}

If $(M,g_{0})$ is a Riemannian manifold, the Ricci flow is an equation for the evolution of $g_{0}$:
\begin{equation}\label{RFequation}
	\partial_{t}g(t)=-2Ric(g(t)), \text{ } g(0)=g_{0}.
\end{equation}
The fixed points of the Ricci flow are Ricci flat metrics. Thus it is natural to ask about the stability of such metrics. This paper will be concerned with the stability of the family of complete Ricci flat Taub-Bolt metrics \cite{Bolt} on the non-compact $\mathbb{CP}^2\setminus\{\text{pt}\}$. All members of the Taub-bolt family are isometric up to diffeomorphism, and so, due to the behaviour of the Ricci flow under the action of diffeomorphism and scaling, the stability, under the Ricci flow, of one member implies the stability of the other members. Thus, this paper will choose a member and say that it is `the' Taub-Bolt metric. 
The main result of this paper is the following dynamical instability result.
\begin{theorem}\label{mainthm}
	There exists a non-trivial ancient solution to the Ricci flow coming out of the Taub-Bolt metric.
\end{theorem}

More precisely, we will show that there is a non-trivial ancient solution to the Ricci flow such that for all $k\geq 0$, we have $g(t)\rightarrow g_{\mathrm{Bolt}}$ (modulo diffeomorphisms) in $C^k$ as $t\rightarrow -\infty$. Thus, Theorem \ref{mainthm} is an improvement of the $L^2$-instability result for Taub-Bolt \cite{hughes2024l2instabilitytaubboltmetricricci} established previously  proven by the author.

On compact manifolds the stability of Ricci flat metrics has been studied and is well understood.
 Building upon the work of Sesum \cite{Ses1},  Haslhofer \cite{Has1} shows that dynamical stability and linear stability of integrable Ricci flat metrics are equivalent. The latter will be determined by the non-positivity of the Lichnernowicz Laplacian associated with the metric.

\begin{definition}\label{LichLap}
	Let $(M,g)$ be a Riemannian manifold. The operator
	\begin{equation*}
		\Delta^{L}_{g}:C^{\infty}(\mathrm{Sym}^2(T^{*}M)) \rightarrow C^{\infty}(\mathrm{Sym}^2(T^{*}M)),
	\end{equation*}
	defined by 
	\begin{equation*}
		\Delta_{g}^{L}h_{ij}=\Delta_{g}h_{ij}+2g^{kp}g^{lq}R_{iklj}h_{pq},
	\end{equation*}
	is called the Lichnernowicz Laplacian, where $R_{iklj}$ are the components of the Riemann curvature tensor.
\end{definition}
\begin{definition}[\cite{Ses1} Definition 2]\label{leqmeaning}
	A Ricci flat metric $g$ on a manifold $M$ is linearly stable if $\Delta_{g}^{L} \leq 0$. That is to say $\bigintsss_{M} \langle \Delta_{g}^{L}h, h \rangle dV_{g} \leq 0$, for $h$ such that the integral is well defined. We say that $g$ is linearly unstable if it is not linearly stable, and write $\Delta_{g}^{L} \not\leq 0$.
\end{definition}

 The following two theorems relate linear stability to dynamical stability of Ricci flat metrics on compact manifolds.
The following theorem has been proven by Sesum \cite{Ses1} and Haslhofer \cite{Has1}.
\begin{theorem} 
	Let $g$ be a Ricci flat metric on a closed manifold $M$ and $k\geq 3$. If $g$ is linearly stable and integrable then for every $C^{k,\alpha}$-neighbourhood $U$ there exists a $C^{k+2,\alpha}$-neighbourhood $V \subset U$ such that Ricci flow starting at any metric in $V$ exists for all time and converges (modulo diffeomorphisms)  to a Ricci flat metric in $U$.
\end{theorem}

\begin{theorem} [\cite{Has1}, Theorem 1.2] \label{Has1}
	Let $g$ be a Ricci flat metric on a closed manifold $M$. If $g$ is linearly unstable, then there exists a non-trivial ancient solution to the Ricci flow with $g(t) \rightarrow g$ (modulo diffeomorphisms)  as $t \rightarrow -\infty$. 
\end{theorem}

To date all known compact Ricci flat manifolds satisfy $\Delta_{g}^{L} \leq 0$. The positive mass problem for Ricci flat manifolds \cite{Cao1} asks if there is a linearly unstable compact Ricci flat metric.
The Taub-Bolt metric, denoted by $g_{\mathrm{Bolt}}$, has been shown to be linearly unstable by Biquard and Ozuch \cite{biquard2025instabilityconformallykahlereinstein}. In this paper we give another, more elementary and explicit, proof of linear instability and prove the following stronger statement.

\begin{theorem}\label{2ndthm}
	There exists a traceless and divergence-free $h\in C^{\infty}(\text{Sym}^2(T^{*}M)) \cap H^1(\mathbb{CP}^2\setminus\{\text{pt}\}, g_{\mathrm{Bolt}})$ solving the equation
	\begin{equation}\label{lambdaeq}
	\Delta^L_{g_{\mathrm{Bolt}}}h =\lambda h \qquad \lambda>0.
\end{equation}
	In particular,  
	\begin{equation}
		\int_{M} \langle \Delta^{L}_{g_{\mathrm{Bolt}}}h, h \rangle_{g_{\mathrm{Bolt}}} dV_{g_\mathrm{Bolt}} >0.
	\end{equation}
\end{theorem}

The construction of the Ricci flow of Theorem \ref{mainthm} will only be possible because of the solution $h$ of equation (\ref{lambdaeq}).

  Theorem \ref{Has1} holds for compact manifolds, and it is not yet known whether a similar result holds in general for non-compact manifolds. However, recently an analogous result has been proven for asymptotically locally Euclidean Ricci flat metrics \cite{Alix2}. Recently, Kim and Ozuch \cite{kim2025ricciflowalfmanifolds} proved the Taub-Bolt metric is dynamically in the sense that there is Ricci flow with initial condition arbitrarily close to the Taub-Bolt metric which does not flow back to Taub-Bolt, using a Perelman's $\lambda$-functional adapted to the ALF setting. The proof of Theorem \ref{mainthm} does not use their $\lambda$-functional, but instead obtains the ancient solution via a hands-on approach. Also, Takahashi \cite{Tak1} constructs an ancient solution coming out of the ALF Euclidean Schwarzschild metric.
  
  This paper is organised as follows. Section \ref{TB} introduces the Taub-Bolt metric, Section \ref{NBS} shows the linear instability of the Taub-Bolt and in particular proves Theorem \ref{2ndthm}. Finally, Section \ref{construct} constructs the ancient solution of Theorem \ref{mainthm}.

$\mathbf{Acknowledgements.}$ I would like to thank my
supervisor Jason Lotay for his support and advice. Also, many thanks to Tristan Ozuch for useful discussions concerning this paper. In particular, for pointing out that the unstable eigentensor appearing in the first version of this paper had not been correctly put in gauge. This research is supported by a scholarship from EPSRC (grant number EP/W524311/1).

\section{Taub-Bolt}\label{TBsect}
The family of Taub-Bolt metrics on $\mathbb{CP}^2\setminus\{\text{pt}\}$  was first discovered by Page \cite{Bolt}, and consists of all the metrics $\alpha g_{\mathrm{Bolt},n}$ for $\alpha>0$, where $g_{\mathrm{Bolt},n}$ can be written in a local frame as:
\begin{equation}\label{TB}
	g_{\mathrm{Bolt},n}=\frac{(r+n)(r+3n)}{r\left(r+\frac{3n}{2}\right)}dr^{2}+ 4(r+n)(r+3n)\left(\sigma_{1}^{2}+\sigma_{2}^{2}\right)+ 16n^{2}\frac{r\left(r+\frac{3n}{2}\right)}{(r+n)(r+3n)}\sigma_{3}^2,
\end{equation}
for $r>0$ and where $\sigma_{1}, \sigma_{2}, \sigma_{3}$ is a Milnor frame on $S^3$.
We note that the metrics $g_{\mathrm{Bolt},n}$ and $g_{\mathrm{Bolt},m}$ are isometric up to scaling. Indeed, if  we consider the diffeomorphism $\phi(r)=\frac{n}{m}r$ of $\mathbb{CP}^2\setminus\{\text{pt}\}$ then $\phi^{*}g_{\mathrm{Bolt},n}=\frac{n^2}{m^2}g_{\mathrm{Bolt},m}$.

It will be convenient for later to calculate the Christoffel symbols, curvature components and sectional curvatures of the Levi-Civita connection associated to $g_{\mathrm{Bolt},n}$. Define 
\begin{align*}\label{abc}
	a(r)&=\frac{(r+3n)(r+n)}{r\left(r+\frac{3n}{2}\right)},\\
	 b(r)&=4(r+3n)(r+n),\\
	c(r)&=16n^2\frac{r\left(r+\frac{3n}{2}\right)}{(r+3n)(r+n)},
\end{align*}
so that
\begin{equation*}
	g_{\mathrm{Bolt},n}=a(r) dr^2+ b(r)\left(\sigma_{1}^2+\sigma_{2}^2\right)+c(r)\sigma_{3}^2.
\end{equation*}
We have the following Christoffel symbols:
\begin{equation}\label{Christ}
	\begin{aligned}
		&\Gamma_{rr}^{r}= \frac{\partial_{r}a}{2a},\quad \Gamma_{1r}^{1}=\Gamma_{r1}^{1}=\Gamma_{2r}^{2}=\Gamma_{r2}^{2}=\frac{\partial_{r}b}{2b},\\
		&\Gamma_{3r}^{3}=\Gamma_{r3}^{3}=\frac{\partial_{r}c}{2c},\quad \Gamma_{11}^{r}=\Gamma_{22}^{r}=-\frac{\partial_{r}b}{2a},\quad \Gamma_{33}^{r}=-\frac{\partial_{r}c}{2a}\\
		&\Gamma_{32}^{1}=-\Gamma_{31}^{2}=\frac{c}{b}-2,\quad \Gamma_{12}^{3}=-\Gamma_{21}^{3}=1,\quad \Gamma_{23}^{1}=\frac{c}{b}=-\Gamma_{13}^{2}=\frac{c}{b}.
	\end{aligned}
\end{equation}
The non-zero curvature terms are the following:
\begin{equation}\label{R's}
	\begin{aligned}
		&R_{r11r}=R_{r22r}=-\frac{\partial_{rr}b}{2}+\frac{\partial_{r}b \partial_{r}a}{4a}+\frac{(\partial_{r}b)^2}{4b}= -n\frac{5r^3+18nr^2+27n^2r+18n^3}{r\left(r+\frac{3n}{2}\right)(r+3n)(r+n)},\\
		&R_{r33r}=-\frac{\partial_{rr}c}{2}+\frac{\partial_{r}c \partial_{r}a}{4a}+\frac{(\partial_{r}c)^2}{4c}= 8n^3 \frac{5r^3+18nr^2+27n^2r+18n^3}{(r+3n)^3(r+n)^3}\\
		& R_{1221}=-\frac{(\partial_{r}b)^2}{4a}+4b-3c=8n^3 \frac{5r^3+18nr^2+27n^2r+18n^3}{(r+3n)(r+n)}  ,\\
		& R_{1331}=R_{2332}=-\frac{\partial_{r}c \partial_{r}b}{4a}+\frac{c^2}{b}=-16n^3 \frac{r\left(r+\frac{3n}{2}\right)(5r^3+18nr^2+27n^2 r+18n^3)}{(r+3n)^3(r+n)^3}\\
		& R_{0123}= R_{0231}= \frac{-1}{2} R_{0312}=  \frac{1}{2}(\partial_{r}c-\frac{c}{b}\partial_{r}b)= -4n \frac{4r^3+9nr^2-9n^3}{(r+3n)^2(r+n)^2}.
	\end{aligned}
\end{equation}

Thus, the sectional curvatures are the following:
\begin{equation} \label{K}
	K_{12}=K_{r3}=-2K_{1r}=-2K_{2r}=-2K_{13}=-2K_{23}=  \frac{5r^3+18nr^2+27n^2r+18n^3}{2n(r+3n)^3(r+n)^3} .
\end{equation}

\section{Linear instability of Taub-Bolt} \label{NBS}
For the rest of this paper set $M=\mathbb{CP}^2\setminus\{\text{pt}\}$. In this section we will define and prove the linear instability of members of the Taub-Bolt family under the Ricci flow. 
We will say that $g_{\mathrm{Bolt},n}$ is linearly unstable if $\Delta^{L}_{g_{Bolt},n} \not\leq 0$, i.e. there exists a traceless and divergence-free $h\in C^{\infty}(\text{Sym}^2(T^{*}M))$ with $h$ and $\nabla h$ of bounded $L_{\mathrm{Bolt},n}^2$ norm and satisfying
\begin{equation}\label{unstabeineq}
	\int_{M} \langle \Delta^{L}_{g_{\mathrm{Bolt},n}}h, h \rangle dV_{\mathrm{Bolt},n} >0.
\end{equation}
Since the metrics $g_{\mathrm{Bolt},n}$ and $g_{\mathrm{Bolt},m}$ are isometric up to scaling, to determine the linear stability of $g_{\mathrm{Bolt},n}$, it suffices to determine the linear stability of $g_{\mathrm{Bolt},1}$. For the rest of this paper we denote $g_{\mathrm{Bolt},1}$ by $g_{\mathrm{Bolt}}$.

After that the existence of a smooth traceless and divergence-free $h\in H^1(M, g_{\mathrm{Bolt}})$ satisfying 

\begin{equation}\label{solution}
	\Delta_{g_{\mathrm{Bolt}}}^{L}h=\lambda h, \qquad \lambda>0,
\end{equation}
 will be shown. This will be achieved by finding a minimiser of the functional 
 \begin{equation}
 	\bold{a}(h)=\int _{M} \lvert \nabla h \rvert ^2- 2R^{\mu \alpha \beta \nu}h_{\alpha \beta}h_{\mu \nu},
 \end{equation}
 among traceless and divergence-free $h\in \mathbb{H}^1$ with unit $L^2(M, g_{\mathrm{Bolt}})$ norm, where $\mathbb{H}^1$ is a weighted Sobelev space with norm 
 	\begin{equation*}
 	\lVert f \rVert_{\mathbb{H}^1}= \left(\int_{M} \lvert \nabla f \rvert^2\right)^{\frac{1}{2}}+ \left(\int_{M} \frac{\lvert f \rvert^2}{(r+1)^2}\right)^{\frac{1}{2}},
 \end{equation*}
 for all $f \in C^{\infty}_{c}(\mathrm{Sym}^2(T^{*}M))$. The equation (\ref{solution}) has been shown to have a solution numerically in \cite{Cla2}, \cite{Young1}. However no analytic proof of a solution is given.
 We will then show that in fact this minimiser is all in $W^{k,2}(M, g_{\mathrm{Bolt}})$ and $L^\infty(M, g_{\mathrm{Bolt}})$. Furthermore, we will show that the norms of $h$ and its derivatives decay exponentially fast in $r$:
 for every $m\geq 0$ there is a $C=C(m), \alpha=\alpha(m)>0$ such that 
 $$|\nabla^m h|_{g_{\mathrm{Bolt}}}(r)\leq Ce^{-\alpha r}.$$
 This will be achieved by showing that the function
 $$I(r)= |h|^2(r) \sqrt{det((g_{\mathrm{Bolt}})_{ij})}\coloneqq |h|^2(r)\nu(r),$$
 which measures the growth of $|h|$, decays exponentially in $r$. Proving this will involve a frequency function
 $$U(r)= \frac{(logI)'}{2}+ \frac{\partial_r \nu}{2\nu},$$
analogous to \cite{Almgren}. It will be demonstrated that $U(r)$ is bounded above by a negative constant proportional to $\sqrt{\lambda}$. Since $\frac{\partial_r \nu}{2\nu}(r) \rightarrow 0$ as $r\rightarrow \infty$, $U$ measures the exponential decay rate of $\sqrt{I}$.
 The notion of a frequency function has been used by many to find growth bounds on solutions to differential equations. For example, in Bernstein \cite{Bern} to study the asymptotic structure of ends of shrinkers for mean curvature flow, and in Colding and Minicozzi \cite{mini1} and \cite{colding2021optimalgrowthboundseigenfunctions} where polynomial growth bounds of $L^2$ eigenfunctions of a Ornstein-Uhlenbeck operator and drift Laplacian respectively are shown.

 To construct the ancient solution of Theorem \ref{mainthm}, Section \ref{construct} will first construct an ancient solution to the Ricci-de Turck flow, then pulling back by a one-parameter family of diffeomorphisms will yield an ancient to the Ricci flow. The advantage of working with the Ricci-de Turck flow rather than the directly with the Ricci flow is that, unlike the Ricci flow, the Ricci-de Turck flow is strictly parabolic, with linearisation about $g_{\mathrm{Bolt}}$ given by 
 $$\partial_t h =\Delta^L_{g_\mathrm{Bolt}} h.$$
 To ensure that the non-trivial Ricci de-Turck flow solution of Section \ref{proofofmain} pulls back to a non-trivial Ricci flow solution, and not just a diffeomorphism and rescaling of $g_{\mathrm{Bolt}}$, we will gauge fix by requiring that $h$ is traceless and divergence-free. 
    
\subsection{Ricci-de Turck flow}

Recall the Ricci flow equation 
\begin{equation}\label{RFeq}
	\partial_{t}g(t)=-2\text{Ric}(g(t)), \text{ } g(0)=g_{0}.
\end{equation}
Suppose $g_{0}$ is Ricci flat. It is shown in \cite{Top1} that the linearisation of $-2\text{Ric}_{g}$ about $g_0$ is
\begin{equation}\label{Riccilin}
	\mathcal{D}_{g_{0}}(-2\text{Ric}_{g}) \coloneqq \frac{d}{ds}\rvert_{s=0}(-2\text{Ric}_{g_{0}+sh})= \Delta_{g_{0}}^{L}h+ \mathcal{L}_{W_{g_{0}}(h)}g_0,
\end{equation}
where $h\in C^{\infty}(\text{Sym}^2(T^{*}M))$ and $W_{g_{0}}(h)$ is  the vector field defined by
\begin{equation*}
	g_{0}(W_{g_{0}}(h), \cdot)= -\text{tr}_{g_{0}}\nabla^{g_{0}}\left(h-\frac{1}{2}\left(\text{tr}_{g_{0}}h\right)g_{0}\right).
\end{equation*}
 
The operator $\mathcal{D}_{g_{0}}(-2\text{Ric}_{g})$ is not elliptic. This can be seen by considering any Ricci flat metric. The linearisation about this Ricci flat metric would have a finite dimensional kernel if the Ricci flow equation was strictly parabolic. But the pullback of a Ricci flat metric under any diffeomorphism is still Ricci flat. So the kernel must be of infinite dimension, a contradiction.  So instead of the Ricci flow we will consider the Ricci-de Turck flow, which is given by the following equation:
\begin{equation}\label{DeTurck}
	\partial_{t} g(t)= -2\text{Ric}(g(t))+ \mathcal{L}_{V(g(t),g_{0})}g(t), \text{ }g(0)=g_{0},
\end{equation}
where $V(g(t),g_{0})$ is a time dependent vector field satisfying
\begin{equation*}
	g(V(g(t),g_{0})(t), \cdot)= -\text{tr}_{g(t)}\nabla^{g(t)} g_{0}-\frac{1}{2}\nabla^{g(t)} \text{tr}_{g(t)}g_{0}.
\end{equation*}
In local coordinates, $V(g(t),g_{0})^{k}=g(t)^{ij}(\Gamma^{k}_{ij}(g(t))-\Gamma^{k}_{ij}(g_{0}))$. In \cite{Shi1}, it is shown that solutions of equation (\ref{RFeq}) and (\ref{DeTurck}) are equivalent up to a time dependent diffeomorphism: If $\phi_{t}$ is the family of diffeomorphisms of $M$ associated with the vector field $V(t)$, meaning $\partial_{t}\phi(p)=V(t,p)$ for all time $t$ and points $p\in M$ and $\phi_{0}=Id_{M}$, then if $g(t)$ satisfies (\ref{RFeq}), we have that $\phi_{t}^{*}g(t)$ satisfies (\ref{DeTurck}). 

The linearisation of $-2\text{Ric}(g)+ \mathcal{L}_{V(g,g_{0})}g$ about a Ricci flat metric $g_{0}$ is
\begin{equation}\label{lin}
	\frac{d}{ds}\rvert_{s=0} (-2\text{Ric}(g_{0}+sh)+ \mathcal{L}_{V(g_{0}+sh,g_{0})}(g_{0}+sh)=\Delta^{L}_{g_{0}} h.
\end{equation}
A proof of (\ref{lin}) can be found in \cite{Alix}. From Definition \ref{LichLap}, it is clear that the Lichnernowicz Laplacian $\Delta_{g_{0}}^{L}$ is elliptic. Furthermore, critical points of a certain functional, which we introduce in (\ref{afuntional}) below are weak solutions to $\Delta^{L}_{g_{\mathrm{Bolt}}} h=\lambda h$ for $\lambda \in \mathbb{R}$, as we explain in Subsection \ref{var}.

\subsection{Functional $\mathbf{a}$}\label{var}
 For the rest of Section \ref{NBS}, $g_{\mathrm{Bolt}}$ will be our background metric, and so all connections, norms, etc.~are with respect to $g_{\mathrm{Bolt}}$. In the rest of Section \ref{NBS} we suppress the volume form in integrals, and $r$ refers to the coordinate in equation (\ref{TB}). In what follows we adapt the variational approach used in \cite{Tak1}. 

\begin{definition}\label{H^1}
	Let $\lVert \cdot \rVert_{\mathbb{H}^1}$ be a norm on $C^{\infty}_{c}(\mathrm{Sym}^2(T^{*}M))$ given by
	\begin{equation*}
		\lVert f \rVert_{\mathbb{H}^1}= \left(\int_{M} \lvert \nabla f \rvert^2\right)^{\frac{1}{2}}+ \left(\int_{M} \frac{\lvert f \rvert^2}{(r+1)^2}\right)^{\frac{1}{2}}.
	\end{equation*}
	for all $f \in C^{\infty}_{c}(\mathrm{Sym}^2(T^{*}M))$.
	
	Furthermore, we define $\mathbb{H}^1$ to be the closure of $C^{\infty}_{c}(\mathrm{Sym}^2(T^{*}M))$ with respect to $\lVert \cdot \rVert_{\mathbb{H}^1}$. We also write $\lVert h \rVert_{2}= \left(\int_{M} \lvert h \rvert^2\right)^{\frac{1}{2}}$ where $h\in \mathbb{H}^1$ for which the integral exists . Finally, let $\mathbb{S}^1\coloneqq \left\{ h \in\mathbb{H}^{1}: h \text{ is diagonal, radially symmetric}, \hspace{1mm} \mathrm{Tr}(h)=0,\hspace{1mm} \mathrm{div}(h)=0, \hspace{1mm} \lVert h\rVert_{2}=1 \right\}$. Here diagonal means with respect to $dr, \sigma_{1}, \sigma_{2}, \sigma_{3}$.
\end{definition}

Define the following functional
\begin{equation}\label{afuntional}
	\bold{a}(h)=\int _{M} \lvert \nabla h \rvert ^2- 2R^{\mu \alpha \beta \nu}h_{\alpha \beta}h_{\mu \nu}.
\end{equation}
for $h\in \mathbb{H}^1$. 
	Using the curvature components (\ref{R's}), it can be checked that there is a constant $C>0$ such that
\begin{equation*}
	\left|R^{\mu \alpha \beta \nu}\right| \leq \frac{C}{(r+1)^3},
\end{equation*} 
for all $\mu, \alpha, \beta, \nu$. Thus, $\mathbf{a}(h)$ exists for all $h\in \mathbb{H}^1$.
  We now consider critical points of $\bold{a}$.  Define $$M_{R} \coloneqq M \cap \{r\leq R\},$$ and denote by $\#:T^{*}M \rightarrow TM$ the map defined by $\alpha=g_{\mathrm{Bolt}}(\alpha^{\#},\cdot)$, for all 1-forms $\alpha$. Let $h\in C^{\infty}(\text{Sym}^2(T^{*}M)) \cap \mathbb{H}^1$ and  $f\in C^{\infty}_{c}(\text{Sym}^2(T^{*}M))$.
Assuming the integration by parts formula 
\begin{equation}\label{integbyparts}
	\int_{M} \langle \nabla f, \nabla h \rangle + \langle f, \Delta h \rangle=0,
\end{equation}
we have, 
\begin{equation*}
	\begin{split}
		\frac{d}{ds}\rvert_{s=0} \bold{a}(h+sf)&=\frac{d}{ds}\rvert_{s=0} \int _{M} \lvert \nabla h+sf \rvert ^2- 2R^{\mu \alpha \beta \nu}(h+sf)_{\alpha \beta}(h+sf)_{\mu \nu}\\
		&= 2\int_{M} \langle \nabla h, \nabla f \rangle- 2R^{\mu \alpha \beta \nu}h_{\alpha \beta}f_{\mu \nu}\\
		&= -2\int_{M} \langle \Delta h, f \rangle+ 2R^{\mu \alpha \beta \nu}h_{\alpha \beta}f_{\mu \nu}\\
		&= -2\int_{M} \langle \Delta^{L} h, f \rangle, \quad \text{ using Definition }\ref{LichLap}. 
	\end{split}
\end{equation*}
Thus, critical points of $\bold{a}$ when restricted to $h\in \mathbb{S}^1$ are weak solutions to $\Delta^{L} h=\lambda h$ for $\lambda \in \mathbb{R}$.

 We now prove (\ref{integbyparts}). Using the identity (which is true for a general metric), 
\begin{equation*}
	\text{tr}_{g_{\mathrm{Bolt,1}}} \nabla (\langle f, \nabla_{\cdot}h  \rangle)=\langle \nabla f, \nabla h \rangle+ \langle f, \Delta h \rangle,
\end{equation*}
where $\langle f, \nabla_{\cdot}h  \rangle$ is the 1-form satisfying $\langle f, \nabla_{\cdot}h  \rangle(X) =\langle f, \nabla_{X}h \rangle$, and so it follows from the divergence theorem that  
\begin{equation}\label{because}
	\begin{split}
		\int_{M_{R}} \langle \nabla f, \nabla h \rangle + \langle f, \Delta h \rangle &= \int_{M_{R}} \text{tr}_{g_{\mathrm{Bolt,1}}} \nabla (\langle f, \nabla_{\cdot} h \rangle)\\
		&= \int_{M_{R}} \text{div}(\langle f, \nabla_{\cdot} h \rangle^{\#})\\
		&= \int_{\partial M_{R}} \langle \langle f, \nabla_{\cdot} h \rangle^{\#}, n_{R} \rangle\\
		& \rightarrow 0 \text{ as } R \rightarrow \infty,
	\end{split}
\end{equation}
where $n_{R}$ is the unit normal to $\partial M_{R}$. The last line in (\ref{because}) is because
\begin{equation*}
	\left\lvert \langle \langle f, \nabla_{\cdot}h \rangle^{\#}, n_{R} \rangle \right\rvert \leq \left\lvert \langle f, \nabla_{\cdot}h \rangle^{\#} \right\rvert \leq \lvert f \rvert \lvert \nabla h \rvert,
\end{equation*}
and so
\begin{equation*}
	\left\lvert \int_{M} \langle \langle f, \nabla_{\cdot}h \rangle^{\#}, n_{R} \rangle \right\rvert \leq \int_{M} \left\lvert \langle \langle f, \nabla_{\cdot}h \rangle^{\#}, n_{R} \rangle \right\rvert \leq \int_{M} \lvert f \rvert \lvert \nabla h \rvert \leq \lVert \nabla h \rVert_{2} \lVert f \rVert_{2}  < \infty.
\end{equation*}
Hence $\int_{\partial M_{R}} \langle \langle f, \nabla_{\cdot}h \rangle^{\#}, n_{R} \rangle \rightarrow 0$ as $R\rightarrow \infty$. Taking the limit $R\rightarrow \infty$ in equation (\ref{because}) gives the formula (\ref{integbyparts}).

\subsection{Linear instability of the Taub-Bolt family}
We say that $g_{\mathrm{Bolt}}$ is linearly unstable if $\Delta^{L} \not\leq 0$ on traceless and divergence-free tensors, i.e there exists a $h\in  \mathbb{S}^1$ such that 
	\begin{equation*}
		\int_{M} \langle \Delta^{L}h, h \rangle >0.
	\end{equation*}

\begin{theorem}\label{<0}
	There exists  $h\in \mathbb{S}^1$ such that $\bold{a}(h)<0$, and so $g_{\mathrm{Bolt}}$ is linearly unstable under the Ricci flow.
\end{theorem}

\begin{proof}
	The proof will proceed as follows. We will find a diagonal, radially symmetric, traceless and divergence-free symmetric two tensor $h$ such that $\bold{a}(h)<0$, and  $h$ is $L^2$ so that after rescaling $h\in \mathbb{S}^1$.  It will be the case that $h$ will have the orthogonal decomposition $$h=h_0-\mathrm{div}_0^*(\omega),$$ for some traceless divergence-free $h_0$ and 1-form $\omega$, 
	where $\mathrm{div}_0^*(\omega)= \frac{1}{2}\mathcal{L}_{\omega^{\#}}g$ is the traceless part of the adjoint of the divergence operator. The tensor $h$ will be obtained by  first constructing $h_0$ and then finding an $\omega$ such that $\mathrm{div}(h)=\mathrm{div}(\mathrm{div}_0^*\omega)$.
	
	Denote the space of traceless symmetric two tensors by $Sym_0^2(T^*M)$.
It will be useful to define the operator $P$ acting on $Sym_0^2(T^*M)$ by  $$P= \Delta^L+\frac{4}{3}\mathrm{div}_0^* \mathrm{div}.$$
	
	In \cite{biquard2025instabilityconformallykahlereinstein}, the following identities are shown,
	$$\Delta^L \mathrm{div}_0^*= \mathrm{div}_0^* B \mathrm{div}^*, \qquad \mathrm{div} \Delta^L= B\mathrm{div}^* \mathrm{div},$$
for $Sym_0^2(T^*M)$,	where $B=\mathrm{div} +\frac{1}{2}d\mathrm{Tr}$ is the Bianchi operator. It follows that $P$ preserves the orthogonal decomposition of elements in $Sym_0^2(T^*M)$ into $ker(\mathrm{div}) \oplus \mathrm{im}(\mathrm{div}^*_0).$
	Thus, since $h$ is also divergence-free, we have 		
	$$\bold{a}(h)=\langle \Delta^L h, h \rangle = \langle P(h), h \rangle= \langle P(h_0), h_0 \rangle -\langle P(\mathrm{div}_0^*(\omega)),  \mathrm{div}_0^*(\omega) \rangle.$$
	
	The second term above is non-negative. Indeed, Delay \cite{MR4730427} 
	 proves that for a compactly-supported 1-form $\tilde{\omega}$, 
	 \begin{equation*}
	 \langle P(\mathrm{div}_0^*(\tilde{\omega})),  \mathrm{div}_0^*(\tilde{\omega}) \rangle = \frac{1}{6}||d^* d \tilde{\omega} ||_{L^2} \geq 0.
	 \end{equation*}
	The fact that $\omega$ decays suitably fast in $r$, the cut-off function method \cite[Appendix A]{Erwan2} can be used to show that,
	$$\langle P(\mathrm{div}_0^*(\omega)),  \mathrm{div}_0^*(\omega) \rangle = \frac{1}{6}||d^* d \omega ||_{L^2}.$$
	Hence it suffices to show that  $$\langle P(h_0), h_0)\rangle <0.$$
	
	We have,
	\begin{align*}
		\langle P h_0, h_0 \rangle &= \langle \Delta^L h_0, h_0 \rangle +\frac{4}{3}\langle \mathrm{div}_0^* \mathrm{div} h_0, h_0 \rangle\\
		& = \langle \Delta^L h_0, h_0 \rangle +\frac{4}{3}\langle  \mathrm{div} h_0,\mathrm{div}  h_0 \rangle\\
		&= \mathbf{a}(h_0)+ \frac{4}{3}\langle  \mathrm{div} h_0,\mathrm{div}  h_0 \rangle.
	\end{align*}

We now construct $h_0$.	
	Recall the functions
	\begin{equation*}
		a(r)=\frac{(r+3)(r+1)}{r(r+\frac{3}{2})},\quad b(r)=4(r+3)(r+1), \quad c(r)=16\frac{r\left(r+\frac{3}{2}\right)}{(r+3)(r+1)},
	\end{equation*}
	defined in Section \ref{TBsect}.  
	Define $k\in C^{\infty}(\text{Sym}^2(T^{*}M))$  by
	\begin{equation*}
		k=a dr^{2}-b(\sigma_{1}^2+\sigma_{2}^2)+c\sigma_{3}^2.
	\end{equation*}
	Then, using (\ref{Christ}), we have
	\begin{equation*}
		\nabla_{1} k_{1r}= \nabla_{1}k_{r1}=\nabla_{2} k_{2r}= \nabla_{2}k_{r2}= \frac{-\partial_{r}b}{2},
	\end{equation*}
	\begin{equation*}
		\nabla_{1}k_{23}=\nabla_{1}k_{32}= -\nabla_{2}k_{13}=-\nabla_{2}k_{31}=c.
	\end{equation*}
	So 
	\begin{equation*}
		\lvert \nabla k \rvert ^2= \frac{(\partial_{r}b)^2}{ab^2}+\frac{4c}{b^2}=\frac{4r\left(r+\frac{3}{2}\right)(2r^2+8r+7)}{(r+3)^3(r+1)^3},
	\end{equation*}
	and 
	$$div(k)= -\frac{4r+8}{(r+3)(r+1)}dr.$$
	Define the piecewise smooth function on $M$ by
	\begin{equation*}
		\eta(r)= \frac{r^2}{(r+1)^{12}}.
		\end{equation*}
		Set the traceless tensor $h_0=\eta k\in \mathbb{H}^{1} \cap L^2$. Then we have an orthogonal decomposition,
			\begin{align*}
				\nabla(\eta k)= d\eta \otimes k+ \eta \nabla k,
			\end{align*}
			Hence,
			\begin{align*}
				\lvert \nabla(\eta k) \rvert ^2
				&= \frac{(\partial_{r} \eta)^2}{a}\lvert k \rvert ^2+\eta ^2\lvert \nabla k \rvert ^2\\
				&= 4\frac{(\partial_{r}\eta) ^2}{a}+ \eta ^2\lvert \nabla k \rvert ^2.
			\end{align*}
			At the end of this proof it will be shown that
			\begin{align*}
				\langle P(h_0), h_0 \rangle  =\bold{a}(h_0)+ \frac{4}{3}\langle  \mathrm{div}\, h_0,\mathrm{div}\, h_0 \rangle<0.
			\end{align*}
			The tensor $h_0$ is not divergence-free. Indeed, $\mathrm{div}(h)= \alpha\,dr$ for some function $\alpha$ depending only on $r$. To remedy this, we look for a function $\beta=\beta(r)$ such that
			$$
			\mathrm{div}\big(\mathrm{div}_0^*(\beta\,dr)\big)= \mathrm{div}(h).
			$$
			Then $h_0-\mathrm{div}_0^*(\beta\,dr)$ will be diagonal, radially symmetric, traceless and divergence-free. The ODE for $\beta$ is
			\begin{align*}
				\alpha
				&= \frac{3}{4}\big(\Delta \beta + (\partial_r \Delta(r))\beta\big)\\
				&= \frac{3}{4}\partial_r\!\big(\partial_r\beta + \Delta(r)\beta\big).
			\end{align*}
			Integrating and setting the integration constant equal to zero gives
			$$
			\alpha_1=\partial_r \beta +\Delta(r)\beta,
			$$
			where $\displaystyle \alpha_1(r)= \int_0^r \frac{4}{3}\alpha (r')\,dr'.$
			Using the integrating factor
			\[
			f(r)=\exp\!\Big(\int_0^r \Delta(s)\,ds\Big)
			\]
			and integrating, with constant zero, gives
			\[
			\beta(r)= f(r)^{-1}\int_0^r \alpha_1(s)\,f(s)\,ds.
			\]
			The choice of integration constants imply that $\beta$ decays suitably fast so that $\mathrm{div}_0^*(\beta\,dr) \in \mathbb{H}^1 \cap L^2$.

				To finish, we show  that $\langle P(h_0), h_0 \rangle  =\bold{a}(h_0)+ \frac{4}{3}\langle  \mathrm{div} h_0,\mathrm{div}  h_0 \rangle<0$. We note that the volume form $dV_{g_\mathrm{Bolt},1}= 4nb$.
	         We compute:
\allowdisplaybreaks
			\begin{align*}
			\langle P(h_0), h_0 \rangle&
			=\bold{a}(h_0)+ \frac{4}{3}\langle  \mathrm{div} h_0,\mathrm{div}  h_0 \rangle\\ &=\int_{M} \lvert \nabla(\eta k) \rvert ^2-2\eta ^2 R^{\mu \alpha \beta \nu}k_{\alpha \beta}k_{\mu \nu} + \frac{4}{3}\int_M |\mathrm{div}(\eta k) |^2\\
		&=4\text{Vol}(S^3)\biggl[ \int_{0}^{\infty} \eta ^2 \lvert \nabla k \rvert ^2b dr +\int_{0}^{\infty}\frac{(\partial_{r}\eta)^2 \lvert k \rvert ^2 b}{a}dr \\
		&\hspace{4mm}-2\int_{3}^{\infty} R^{ijji}k_{ii}k_{jj}\eta^2 dr + \frac{4}{3}\int_M |\mathrm{div}(\eta k) |^2\biggr]\\
		&=4\text{Vol}(S^3)\biggl[\int_{0}^{\infty}\left(\lvert \nabla k \rvert ^2-4(K_{r3}+K_{12}-K_{13}-K_{23}-K_{r1}-K_{r2})\right)\eta^2 b dr\\
		& \hspace{60mm}+\int_{0}^{\infty} 4\frac{(\partial_{r}\eta)^2}{a}b dr + \frac{4}{3}\int_0^\infty |\mathrm{div}(\eta k) |^2 b dr\biggr]\\
		&=4\text{Vol}(S^3)\biggl[4\int_{3}^{\infty}\left(\frac{8r^4+4r^3-68r^2-174r-144}{(r+3)^2(r+1)^2}\right) \frac{r^4}{(r+1)^{24}}dr \\
		&\hspace{4mm}+\int_{0}^{\infty}  16\left(\frac{2r-\frac{12r^2}{(r+1)}}{(r+1)^{24}}\right)^2r(r+\frac{3}{2})dr\\
		&\hspace{4mm} + \frac{4}{3}\int_0^{\infty}  4\left(2r-\frac{12r^2}{(r+1)}+2\frac{2r+4}{(r+3)(r+1)}\right)^2r\left(r+\frac{3}{2}\right) \frac{1}{(r+1)^{24}} dr \biggr]\\
		&=  4\text{Vol}(S^3)[-0.00027...+0.000048...+0.000052...]\\
		&<0.
	\end{align*}

	Therefore, $g_{\mathrm{Bolt}}$ is linearly unstable under the Ricci flow.
\end{proof}
\subsection{Existence of a positive eigenvalue of $\Delta^{L}_{g_{\mathrm{Bolt}}}$}
The existence of an $h\in \mathbb{S}^1$ solving $\Delta^{L}h=\lambda h$ with $\lambda>0$, by finding a minimiser of $\bold{a}$ will be now be shown. We adapt the argument given in \cite[Theorem 3.5]{Tak1} and check that we can assume $h$ is traceless and divergence-free.

\begin{theorem}\label{h}
	There exists a non-zero $h\in \mathbb{S}^1$ solving 
	\begin{equation*}
		\Delta^{L}h=\lambda h
	\end{equation*}
	with $\lambda>0$, in the sense that for all $f\in C^{\infty}_{c}(\mathrm{Sym}^2(T^{*}M))$, 
	\begin{equation*}
		\int_{M} \langle \nabla h, \nabla f \rangle-2R^{\mu \alpha \beta \nu}h_{\alpha \beta}f_{\mu \nu}+\lambda \langle h, f \rangle =0.
	\end{equation*}
\end{theorem}

\begin{proof}
	Using the curvature components (\ref{R's}), it can be checked that there is a constant $C>0$ such that
	\begin{equation*}
		\left|R_{\mu \alpha \beta \nu}\sqrt{g^{\mu \mu}g^{\alpha \alpha}g^{\beta \beta}g^{\nu \nu}}\right| \leq \frac{C}{(r+1)^3},
	\end{equation*} 
	for all $\mu, \alpha, \beta, \nu$. Therefore, 
	\begin{equation*}
		\begin{split}
			\left|R^{\mu \alpha \beta \nu}h_{\alpha \beta}h_{\mu \nu} \right|&= \left|R_{\mu \alpha \beta \nu} g^{\mu \mu}g^{\alpha \alpha}g^{\beta \beta}g^{\nu \nu} h_{\alpha \beta}h_{\mu \nu}\right|\\
			&= \left|R_{\mu \alpha \beta \nu}\sqrt{g^{\mu \mu}g^{\alpha \alpha}g^{\beta \beta}g^{\nu \nu}} (\sqrt{g^{\alpha \alpha}g^{\beta \beta}} h_{\alpha \beta})(\sqrt{g^{\mu \mu}g^{\nu \nu}}h_{\mu \nu})\right|\\
			&\leq \frac{C}{2(r+1)^3}(g^{\alpha \alpha}g^{\beta \beta}h_{\alpha \beta}^2+ g^{\mu \mu}g^{\nu \nu}h_{\mu \nu}^2)\\
			&= \frac{C \lvert h \rvert ^2}{(r+1)^3}\\
			&\leq C\lvert h \rvert^2.
		\end{split}
	\end{equation*}
	Thus $\inf_{h \in \mathbb{S}^1} \bold{a}(h)>-\infty$. Take a sequence $h^{(n)}\in \mathbb{S}^1$ such that 
	\begin{equation*}
		\lim_{n \rightarrow \infty} \bold{a}(h^{(n)})= \inf_{\overline{h}\in \mathbb{S}^1} \bold{a}\left(\overline{h}\right). 
	\end{equation*}
	By Theorem \ref{<0}, we can assume $\bold{a}(h^{(n)})<0$ for all $n$. Then
	\begin{equation*}
		\begin{split}
			0> \bold{a}(h^{(n)})&= \int_{M} \lvert \nabla h ^{(n)} \rvert^2-2R^{\mu \alpha \beta \nu}h^{(n)}_{\alpha \beta}h^{(n)}_{\mu \nu} \\
			&\geq \int_{M} \lvert \nabla h ^{(n)} \rvert ^2 - 2C \int_{M} \frac{\lvert h^{(n)} \rvert^2}{(r+1)^{3}}\\
			&\geq \int_{M} \lvert \nabla h ^{(n)} \rvert ^2- 2C \int_{M} \lvert h^{(n)} \rvert^2.
		\end{split}
	\end{equation*}
	Therefore $h^{(n)}$ is bounded in $\mathbb{H}^{1}$. Since $\mathbb{H}^1$ is a Hilbert space the unit ball in $\mathbb{H}^1$ is weakly compact and so, after passing to a subsequence, $h^{(n)}$ is weakly convergent to some $h \in \mathbb{H}^{1}$ with $\liminf_{n \rightarrow \infty} \lVert h^{(n)} \rVert _{\mathbb{H}^{1}} \geq \lVert h \rVert_{\mathbb{H}^{1}}$. 
	
	On any compact set $M_{R} \coloneqq M \cap \{r\leq R\}$ with $R>0$, the norms $\left(\bigintsss_{M_{R}} \lvert \tilde{h} \rvert^2\right)^{\frac{1}{2}}$ and $\left(\bigintsss_{M_{R}}  \frac{\lvert\tilde{h} \rvert^2}{(r+1)^2}\right)^{\frac{1}{2}}$ are equivalent. Hence, using Rellich-Kondrachov, for any $R>0$, the boundness of the $\lVert h^{(n)} \rVert_{\mathbb{H}^{1}, M_{R}}$ implies that after passing to a further subsequence, $h^{(n)}$ is strongly convergent in $L^{2}(M_{R})$. Thus, as $n\rightarrow \infty$,  
	\begin{equation}\label{Rlimit}
		\int_{M_{R}} \lvert h^{(n)} \rvert ^2 \rightarrow \int_{M_{R}} \lvert h \rvert^2, \int_{M_{R}} R^{\mu \alpha \beta \nu}h^{(n)}_{\alpha \beta}h^{(n)}_{\mu \nu} \rightarrow \int_{M_{R}} R^{\mu \alpha \beta \nu}h_{\alpha \beta}h_{\mu \nu}.
	\end{equation}
	 Let $\epsilon>0$. Since $\liminf_{n \rightarrow \infty} \lVert h^{(n)} \rVert_{\mathbb{H}^{1}} \geq \lVert h \rVert_{\mathbb{H}^{1}}$,
	\begin{equation*}
		\int_{M} \lvert \nabla h^{(n)} \rvert^2+ \int_{M} \frac{\lvert h^{(n)} \rvert^2}{(r+1)^2} + \epsilon \geq \int_{M} \lvert \nabla h \rvert^2+ \int_{M} \frac{\lvert h \rvert^2}{(r+1)^2},
	\end{equation*}
	for $n$ sufficiently large. Choose $R>0$ large enough so that 
	\begin{equation}\label{epsilon}
		\int_{M\setminus M_{R}} \frac{\lvert h^{(n)} \rvert^2}{(r+1)^2} \leq  \frac{\lVert h^{(n)} \rVert_{2}^2}{(R+1)^2}\leq \epsilon , \text{ } \int_{M\setminus M_{R}} \frac{\lvert h \rvert^2}{(r+1)^2} \leq  \frac{\lVert h \rVert_{2}^2}{(R+1)^2}\leq \epsilon.
	\end{equation}
	It follows that
	\begin{equation*}
		\int_{M} \lvert \nabla h^{(n)} \rvert^2+ \int_{M_{R}} \frac{\lvert h^{(n)} \rvert^2}{(r+1)^2} + 3\epsilon \geq \int_{M} \lvert \nabla h \rvert^2+ \int_{M_{R}} \frac{\lvert h \rvert^2}{(r+1)^2}.
	\end{equation*}
	Therefore, using (\ref{Rlimit}),
	\begin{equation*}
		\liminf_{n \rightarrow \infty} \int_{M} \lvert \nabla h^{(n)} \rvert^2+ 4\epsilon \geq \int_{M} \lvert \nabla h \rvert^2,
	\end{equation*}
	for all $\epsilon>0$. So 
	\begin{equation}\label{nablabound}
		\liminf_{n \rightarrow \infty} \int_{M} \lvert \nabla h^{(n)} \rvert^2 \geq \int_{M} \lvert \nabla h \rvert^2.
	\end{equation}
	Combining by (\ref{nablabound}) and (\ref{Rlimit}) yields $\bold{a}(h) \leq \lim_{n \rightarrow \infty} \bold{a}(h^{(n)})<0$.
	  By a standard diagonal argument, we can assume that for every $R>0$, $h^{(n)}$ converges strongly in $L^2(M_R)$. 
		 From the strong convergent on $M_R$, it is clear that $h$ is also diagonal, radially symmetric, and traceless. We now show that it is divergence-free by showing that $\nabla h^{(n)}$ converges strongly to $\nabla h$ in $L^2(M)$. 
	
	 Let $\epsilon>0$. We continue with the assumption that for every $R>0$, $h^{(n)}$ converges strongly in $L^2(M_R)$. For $R(\epsilon)=R>0$ large enough, we have (\ref{epsilon}). So
	 \begin{align*}
	\int_M |\nabla h^{(n)}|^2-|\nabla h|^2&=
	\bold{a}(h^{(n)})-\bold{a}(h)+\int_{M_R} 2R^{\mu \alpha \beta \nu}h_{\alpha \beta}h_{\mu \nu}-2R^{\mu \alpha \beta \nu}h^{(n)}_{\alpha \beta}h^{(n)}_{\mu \nu}\\
	& \hspace{4mm }+\int_{M\setminus M_R} 2R^{\mu \alpha \beta \nu}h_{\alpha \beta}h_{\mu \nu}-2R^{\mu \alpha \beta \nu}h^{(n)}_{\alpha \beta}h^{(n)}_{\mu \nu}.
	\end{align*}
	Since $\bold{a}(h^{(n)})\rightarrow \bold{a}(h)$ as $n \rightarrow \infty$ and (\ref{Rlimit}), it follows that for $n$ large enough, 
	$$\left|\int_M |\nabla h^{(n)}|^2-|\nabla h|^2\right| \leq 3\epsilon.$$
	Therefore $|\nabla h^{(n)}| \rightarrow |\nabla h|$ strongly in $L^2(M)$. By the weak convergence $h^{(n)} \rightarrow h$ in $\mathbb{H}^1$, 
	\begin{align*}
	\lim_{n\rightarrow \infty}||\nabla h^{(n)}-\nabla h||_{L^2(M)}^2&= \lim_{n \rightarrow \infty} ||\nabla h^{(n)}||^2_{L^2(M)}+ ||\nabla h||_{L^2(M)}^2- 2\langle \nabla h^{(n)}, \nabla h \rangle \\
	&= \lim_{n \rightarrow \infty} ||\nabla h^{(n)}||^2_{L^2(M)}- ||\nabla h||_{L^2(M)}^2 \\
	&= 0.
	\end{align*}
\end{proof}

\subsection{Properties of $h$}
The following lemma will be needed to prove that $h\in W^{m,2}, L^{m,\infty}, C^{\infty}$. It will also be a vital ingredient in the proof of Theorem \ref{mainthm} in the next section.

\begin{lemma}\cite{Aub}\label{sobelevineq}
	Let $(M',g)$ be a four dimensional Riemannian manifold with  bounded curvature by $K>0$ and positive injectivity radius $\rho>0$.
	If $|k|_g \in W^{m+3,2}(M',g)$, then there is $C=C(m, K, \rho)>0$ such that
	$$|\nabla^m k|_g\leq C ||k||_{W^{m+3,2}(M',g)}.$$
\end{lemma}

 \begin{remark}
  Since $M$ is four dimensional, Sobelev inequalities can only bound the $C^m$ norm of a tensor by its $W^{m+3, 2}$ norm. This differs from \cite{Tak1}, where the inequality \cite[Proposition 4.1]{Tak1} is used, which bounds the $C^m$ norm of a tensor by its $W^{m+1, 2}$ norm. This is main place at which our proof of Theorem \ref{mainthm} and Takahashi's construction of an ancient solution coming out of the Euclidean Schwarzschild metric \cite{Tak1} diverge.
 \end{remark}
We note that $(M, g_{\mathrm{Bolt}})$ has bounded curvature and positive injectivity radius.

\begin{lemma}
	$h \in W^{m,2}, L^{m,\infty}, C^\infty$ for all $m \geq 0$. 
\end{lemma}
\begin{proof}
	Since $h\in W^{1,2}$ is a weak solution to $\Delta_L h=\lambda h$, \cite[Section 6.3]{Evans} shows that if $h \in W^{m,2}$  then 
	$$||h||_{W^{m+2,2}} \leq C(m) (||h||_{W^{m,2}} +||h||_{L^2}).$$
	Therefore an inductive argument proves $h \in W^{m,2}$ for all $m\geq0$, meaning we also have that $h$ is smooth.
	Lemma \ref{sobelevineq} shows that $h\in L^{m,\infty}$ for all $m\geq 0$.  
\end{proof}
Now we show that the norms $h$ and its derivatives decay exponential fast in $r$. The technique used to prove this will be that of a frequency function. 
\begin{proposition}\label{boundsonh}
	For every $m\geq 0$ there is a $C=C(m), \alpha=\alpha(m)>0$ such that 
	$$|\nabla^m h|(r)\leq Ce^{-\alpha r}.$$
\end{proposition}

\begin{proof}
	Set $\nu\coloneqq 16(r+3)(r+1)$. Then $dVol_{g_{\mathrm{Bolt}}}= \nu dr \wedge \sigma_{1} \wedge \sigma_{2}, \wedge \sigma_{3}$. 
	Define for $r>0$,
	\begin{align*}
	I(r)&= |h|^2 \nu,\\
	 D(r)&= \int_{0}^{r} [|\nabla h |^2+\langle \Delta_L h, h \rangle- 2R^{\mu \alpha \beta \nu}h_{\alpha \beta}h_{\mu \nu}] \nu d\tilde{r}= \langle \nabla_{\partial_r}h, h \rangle \nu (r).
	\end{align*}
	Then 
	\begin{align*}
	I'(r)&=2|h|\partial_r |h| \nu + |h|^2\partial_r \nu = 2D+ \frac{\partial_r \nu}{\nu} I, \\
	D'(r)&= [|\nabla h |^2+\langle \Delta_L h, h \rangle- 2R^{\mu \alpha \beta \nu}h_{\alpha \beta}h_{\mu \nu}] \nu \geq \frac{\lambda}{2} |h|^2,
	\end{align*}
	for $r$ large enough since $|Rm|\rightarrow 0$ as $r\rightarrow \infty.$
	Define $$U \coloneqq\frac{D}{I}= \frac{ID'-I'D}{I^2} \geq \frac{\lambda}{2}-2U^2-\frac{\partial_{r} \nu}{\nu}U.$$
	Claim: $U$ is bounded above
	
	Proof of claim: Since $\nu $ is a polynomial in $r$, there exists $K_1>0$ such that $\frac{\partial_r \nu}{\nu} \leq K_1$ for large $r$. Let $K_2$ be the positive solution to $\frac{\lambda}{2}-2x^2
	-K_1x=0$. Suppose that $U(r) \geq K_2$ for some $r'$. Then $U'(r')\leq 0$. Thus, $U(r)\leq K_2$ for large $r$.
	
	Since $\frac{\partial_r \nu}{\nu}(r) \rightarrow 0$ as $r \rightarrow \infty$, 
	$$U'\geq \frac{\lambda}{4}-2U^2,$$
	for large enough, say $r>r_1$. 
	
	Suppose $\frac{\lambda}{2}-2U^2(r_2)>0$ for some $r_2>r_1$. Then 
	$$\frac{U'}{\frac{\lambda}{2}-2U^2(r_2)}\geq 1.$$ Therefore,
	$$\frac{\lambda}{2}-2U^2\geq Ae^r,$$
	for some $A>0$ and $r>r_2$. This is a contradiction for $r$ large enough. Thus, 
	$$U^2\geq \frac{\lambda}{8}.$$
	
	Suppose $U\geq \sqrt{\frac{\lambda}{8}}.$ Then $$\frac{\partial_{r} |h|^2}{2|h|^2}=\frac{\langle \nabla_{\partial_{r}} h, h \rangle}{|h|^2}=U \geq \sqrt{\frac{\lambda}{8}},$$ and so $|h|$ grows exponentially in $r$. This contradicts $h \in L^2$. Therefore, 
	$$U\leq -\sqrt{\frac{\lambda}{8}},$$
	for large $r$. This implies,
	$$\frac{I'}{I}\leq- \sqrt{\frac{\lambda}{2}}+\frac{\partial_{r} \nu}{\nu} \leq -\frac{1}{2}\sqrt{\frac{\lambda}{2}},$$
	for $r$ large enough. Hence,
	$$I\leq Be^{-\sqrt{\frac{\lambda}{2}}r},$$
	for $r$ large. Since $\nu$ has polynomial growth rate, it follows that $|h|(r)$ decays exponentially in $r$.
	
	Now we prove the exponential decay of the higher derivatives of $h$. Since $U\leq -\sqrt{\frac{\lambda}{8}}$, we have 
	$$\partial_{r} |h|^2 \leq -\sqrt{\frac{\lambda }{8}}|h|^2,$$
	and so
	$\partial_r |h|^2$ also decays exponentially in $r$. The equation $\Delta_L h= \lambda h$ implies, by induction, that all
	$$\frac{\partial^m}{\partial r^m} |h|,$$ 
	also decays exponentially for $m\geq 2$. Since $h$ is radially symmetric, $|\nabla^m h|$ decays exponentially for all $m\geq 0$. This completes the proof.	
\end{proof}

\section{Constructing the ancient solution}\label{construct}
The purpose of this section is to prove Theorem \ref{mainthm}, the existence of a non-trivial ancient solution emerging from the Taub-Bolt metric. The method of proof is inspired by the work of Takahashi \cite{Tak1}, who constructs an ancient solution coming out of the Euclidean Schwarzschild metric. 

To obtain such an ancient solution emerging from Taub-Bolt, we will prove the existence of a sequence of solutions $g^n(t)$ to the Ricci-de Turck flow equation defined for $t\in[t_n, N]$, where $\epsilon_n=e^{\lambda t_n}$ and $\epsilon_n \rightarrow 0$ as $n\rightarrow \infty$, $N$ is independent of $n$, and $$g^n(t_n)= g_{\mathrm{Bolt}}+\epsilon_n h.$$ The Arzelà–Ascoli  theorem combined with the $C^4$ bounds resulting from Theorem \ref{l^2long} and the Sobelev inequality \ref{sobelevineq} controlling our solution will yield a subsequence that converges to an ancient solution coming out of the Taub-Bolt metric. We will also check that this ancient solution is not the trivial stationary solution $g_{\mathrm{Bolt}}$.

Constructing the sequence $g^{\epsilon_n}$ will be achieved by controlling the tensor
$$w^n(t)\coloneqq g^n (t)-g_{\mathrm{Bolt}}-e^{\lambda t}h.$$
By the definition of $\epsilon_n$, $w^n(t_n)=0$. Furthermore, the time-evolution equation for $w$ is 
$$\partial_t w^n= \Delta^L_{g_\mathrm{Bolt}}w^n+ O((w^n)^2, e^{\lambda t}).$$
The definition of $\lambda$ and these two facts suggest that there will be a $C>0$, independent of $n$, such that 
$$||w^n||_2 =\left(\int_M | w^n|_{g_\mathrm{Bolt}}^2 dVol_{g_\mathrm{Bolt}}\right)^{\frac{1}{2}} \leq Ce^{\lambda t}.$$
In fact,  Theorem \ref{l^2long} will imply that 
there exists $M_{i,j},T>0$ and $N\leq 0$, all independent of $n$,  such that if 
\begin{equation}\label{introeq}
	\left|\left|\frac{d^j}{dt^j} \nabla^iw^n\right|\right|_2(t) \leq M_{i,j}\delta(t), \qquad i+j\leq 7,
\end{equation}
for all $t\in [t_0, t_0+(k-1)T]$ and $t_0+kT<N$, then $w$ exists for another time $T$ and 
$$\left|\left|\frac{d^j}{dt^j} \nabla^iw^n\right|\right|_2(t) \leq M_{i,j}\delta(t),\qquad i+j\leq 7,$$
for all $t\in[t_0, t_0+kT]$. 
The equation (\ref{introeq}) and the Sobelev inequalities of Lemma $\ref{sobelevineq}$ will give $C^3$ control of $w^n$ which will imply that the curvature of $g^n(t_n+(k-1)T)$ will satisfy
$$|Rm(g^n(t_n+(k-1)T))|_{g^n(t_n+(k-1)T)} \leq K,$$
for $K$ is independent of $k,n$. Standards results for Ricci flow show that $g^{\epsilon_n}(t)$ exists for another time $T$, for $T=T(K)>0$ chosen sufficiently small.

A brief outline of this section is as follows:
Section \ref{groundwork} will set the notation for the rest of Section \ref{construct}, prove a result which will allow us to show that the $H^m$ norms of $w_n$ are actually finite, and analyse the time-evolution equation for $w_n$.
Section \ref{C^mbounds} will prove the $C^m$ bounds of $w_n$ which will be used to control the terms not appearing in the linearisation $\partial_t w_n= \Delta_{\mathrm{Bolt}}^Lw_n$.
Section \ref{L^2bounds} will prove Theorem \ref{l^2long} stated informally above.
Section \ref{proofofmain} will prove Theorem \ref{mainthm}.

\subsection{The groundwork}\label{groundwork}

For the rest of Section \ref{construct},  set $g_0=g_{\mathrm{Bolt}}$. Also, very often we will use the notation that $C(A,B, ...)$  will denote a some constant $C>0$ depending on $A, B, ...$.
The following Theorem gives a lower bound on the maximal existence time of the Ricci-de Turck flow given a bound on the curvature of the initial condition, as well as bounds on the solution during this period of existence.

\begin{theorem}\label{shorttimeests}
	Suppose $g(0)$ satisfies $|Rm(g(0))|\leq K$ for some $K>0$. Then there exists a unique complete solution to the
	Ricci-de Turck equation 
	\begin{align*}
		&\partial_{t} g(t)= -2\text{Ric}(g(t))+ \mathcal{L}_{V(g(t),g(0))}g(t),
	\end{align*}
	where $V(g(t),g(0))$ is a time dependent vector field satisfying
	\begin{equation*}
		g(V(g(t),g(0))(t), \cdot)= -\text{tr}_{g(t)}\nabla^{g(t)} g_{0}-\frac{1}{2}\nabla^{g(t)} \text{tr}_{g(t)}g(0),
	\end{equation*}
	for $t\in[0, T]$ for some $T=T(K)>0$.
	Furthermore, for any $\rho>0$,  $T=T(\rho, K)$ can be made small enough so that 
	\begin{align*}
		&|g(t)-g(0)|_{g(0)}\leq \rho,\\
		&|\overline{\nabla} g |_{g(0)},
		\left|\frac{\partial g}{\partial t} \right|_{g(0)} \leq C
	\end{align*}
	for all $t\in[0,T]$ and some $C=C(K)$.
\end{theorem}

By the bounds of Proposition \ref{boundsonh}, for all $\epsilon<\epsilon'$ for $\epsilon'$ to be determined in the course of the proof of Theorem \ref{mainthm}, $|Rm(g_0+\epsilon h)|\leq \frac{K}{2}\leq K$ for some $K=K(\epsilon')>0$. Fix $\epsilon<\epsilon'$ for the rest of Section \ref{construct}. Then Theorem \ref{shorttimeests} implies that for $g(0)=g_0+\epsilon h$, there exists a $T=T(\epsilon')=T(K)>0$ such the Ricci-de Turck equation has a solution on $t\in[t_0,t_0+T]$ with $g(t_0)=g_0+\epsilon h$ for $\epsilon=e^{t_0 \lambda}$. For the rest of Section \ref{construct} it will be necessary, as we progress through the lemmas, to have $T$ be sufficiently small, but at a minimum it will always be less than or equal to $T(\epsilon')$.

It will be useful to define the following tensors,
\begin{align*}
	&g_\epsilon=\overline{g}= g_0+\epsilon h,\\ 
	& v=g-\overline{g}.
\end{align*}

The components of $g_\epsilon$ and covariant derivates with respect to $g_{\epsilon}$ will be denoted by $\overline{g}_{ij}$ and $\overline{\nabla}$ respectively. Also, the components of $g_0$ and covariant derivates with respect to $g_0$ will be denoted by $\dot{g}_{ij}$ and $\dot{\nabla}$ respectively.   Denote by $|\cdot |, |\cdot |_0, |\cdot |_\epsilon$ the norms on tensors corresponding to $g, \dot{g}, \overline{g}$ respectively, and their $C^k$ norms by $||\cdot||_{C^k}, ||\cdot||_{C^k}^{(0)}, ||\cdot||_{C^k}^{(\epsilon)}$ respectively. We  note that the bounds of Proposition \ref{boundsonh} imply that the norms $|\cdot |_0, |\cdot |_\epsilon$ are equivalent. Furthermore, for $\epsilon<\epsilon'$, $\overline{g}$ has bounded curvature and positive injectivity radius, both uniformly in $\epsilon$.

By \cite{Shi1},   if $g(t)$ is a solution to the Ricci-de Turck flow then, 
\begin{align*}
	\partial_t g_{ij}&=-2R_{ij}+\mathcal{L}_{V(g(t),g_{0})}g_{ij} \\
	&= g^{\alpha\beta}\overline{\nabla}^2_{\alpha\beta}g_{ij}-g^{\alpha\beta}g_{ip}\overline{g}^{pq}\overline{R}_{j\alpha q \beta}-g^{\alpha\beta}g_{jp}\overline{g}^{pq}\overline{R}_{i \alpha q \beta}\\
	&\hspace{4mm}+ \frac{1}{2}g^{\alpha\beta}g^{pq}(\overline{\nabla}_i g_{p\alpha}\overline{\nabla}_j g_{q\beta}+ 2\overline{\nabla}_{\alpha}g_{jp}\overline{\nabla}_q g_{i\beta}-2\overline{\nabla}_\alpha g_{jp}\overline{\nabla}_\beta g_{iq}\\
	&\hspace{46mm}-2\overline{\nabla}_i g_{p\alpha}\overline{\nabla}_\beta g_{iq}-2\overline{\nabla}_i g_{p\alpha}\overline{\nabla}_\beta g_{jp}).
\end{align*}

Since $g(t), g_0, \overline{g}$ are all in the diagonal form with respect to $dr^2, \sigma_1, \sigma_2, \sigma_{3}$ and is radially symmetric,  it follows that $v$ satisfies
\begin{align}\label{veq}
	\partial_{t} v_{ii}=g^{\alpha\alpha}\overline{\nabla}^2_{\alpha\alpha} v_{ii}+A_{ii},
\end{align}
where $$A_{ii}=-2g^{\alpha\alpha}g_{ii}\overline{g}^{ii}\overline{R}_{i\alpha i\alpha},$$ for $i\not= r$, and 
$$A_{rr}= -2g^{\alpha\alpha}g_{rr}\overline{g}^{rr}\overline{R}_{r\alpha r\alpha}+\frac{1}{2}g^{kk}g^{kk}(\overline{\nabla}_r v_{kk})^2.$$
We have
\begin{align*}
	\partial_t |v|_{\epsilon}^2&=2\langle \partial_t v, v \rangle_{\epsilon}\\
	&= 2\langle g^{\alpha\alpha}\overline{\nabla}^2_{\alpha\alpha}v, v \rangle_{\epsilon} + (\overline{g}^{ss}g^{kk})^2(\overline{\nabla}_{s}v_{kk})^2v_{11}+\langle F, v \rangle_{\epsilon} \\
	&=g^{\alpha\alpha}\overline{\nabla}^2_{\alpha\alpha}|v|^2_{\epsilon}-2g^{\alpha\alpha}\langle \overline{\nabla}_{\alpha}v, \overline{\nabla}_\alpha v \rangle_{\epsilon}+ (\overline{g}^{rr}g^{kk})^2(\overline{\nabla}_{r}v_{kk})^2v_{rr}+\langle F, v \rangle_{\epsilon},
\end{align*}
where $F=2g^{\alpha\alpha}g_{ii}\overline{g}^{ii}\overline{R}_{i\alpha i \alpha}.$ 

We check that, at least on a short time interval, $|v|_{\epsilon}$ decays faster than $r^{-l}$ for any $l>0$. In particular, $||v||_2^{(\epsilon)}(t)$ is finite. The proof of the following lemma is inspired by \cite[Proposition 5.7]{Tak1}

\begin{lemma}\label{vdecay}
	Let $T>0$ be sufficiently small. Then 
	$|v|_\epsilon(t)\leq C(l,T)r^{-l}$ for all $l>0$ and $t\in [t_0, t_0+T]$.
\end{lemma}
\begin{proof}
	By Theorem \ref{shorttimeests}, we can make $T$ small enough so that 
	$$-2g^{\alpha\alpha}\langle \overline{\nabla}_{\alpha}v, \overline{\nabla}_\alpha v \rangle_\epsilon+ (\overline{g}^{rr}g^{kk})^2(\overline{\nabla}_{r}v_{kk})^2v_{11}<0.$$
	By making $T$ smaller again, by Theorem \ref{shorttimeests}, we can assume $|v|_\epsilon$ is bounded above so that
	$$\partial_t |v|_{\epsilon}^2\leq g^{\alpha\alpha}\overline{\nabla}^2_{\alpha\alpha}|v|^2_{\epsilon}+\langle F, v \rangle_\epsilon= L(|v|^2_{\epsilon})+\langle F, v \rangle_\epsilon,$$
	with $L$ uniformly elliptic operator on $[t_0, t_0+T]$.
	
	Use the parametric method of \cite{Friedman}  with background metric $g_\epsilon$, to find the fundamental solution $\Phi(x, \eta, t, \tau)$ of $L$. Then 
	$$v(x,t)= \Phi * |\langle F, v \rangle |= \int_{t_0}^{t_0+T} \int_{0}^{\infty} \Phi(x, \eta, t, \tau) \langle F, v \rangle(\eta, \tau) dVol_{\overline{g}}.$$
	Since $|v|$ is bounded on $[t_0, t_0+T]$ and $|Rm|_{\epsilon}\leq Cr(x)^{-3}$, it follows that $|\langle F, v \rangle| \leq Cr(x)^{-3}$. By the estimates  of \cite{Friedman},
	$$|v|^2(x(r),t)\leq Cr^{-3}, \qquad t\in [t_0, t_0+T].$$
	To obtain a stricter bound on $|v|$, we notice that
	$$F=2g^{\alpha\alpha}\overline{g}^{\alpha\alpha}v_{\alpha\alpha}g_{ii}\overline{g}^{ii}\overline{R}_{i\alpha i\alpha}+2g_{ii}\overline{g}^{ii}\overline{R}_{ii}.$$
	And so, by Proposition \ref{boundsonh}, $|\dot{\nabla}^m h|_0 \leq C_{l,m}r^{-l}$ for all $l$, and so, since $\dot{Ric}=0$,
	$$|\overline{Ric}|_\epsilon \leq C_l r^{-l},$$
	for all $l$. So $|\langle F, v \rangle_\epsilon|_\epsilon\leq Cr^{-6}.$ Estimating as above yields the stronger bound $|v|_\epsilon\leq Cr^{-3}$. Repeating this argument we have, by induction, $|v|_\epsilon\leq C(l)r^{-l}$ for all $l$.
\end{proof}
We now check the same for $|\overline{\nabla}^m v|_\epsilon.$

\begin{lemma}
	$|\overline{\nabla}^m v|_\epsilon(t)\leq C(m,l)r^{-l}$ for all $m,l>0$ and $t\in [t_0, t_0+T_0]$.
\end{lemma}
\begin{proof}
	By \cite{Shi1}, $|\nabla^m v|_\epsilon$ is bounded all $m$. By Lemma \ref{vdecay}, $|v|_\epsilon\leq C_{l}r^{-l}$ for all $l$. It follows that  $|\nabla^m v|_\epsilon(t)\leq C_{l,m}r^{-l}$ for all $m,l>0$
\end{proof}

Recall that $g(t)$ is the solution to the Ricci-de Turck flow with initial condition $g(t_0)=g_0+\epsilon h,$ where $\epsilon\coloneqq e^{\lambda t_0}$,
$\dot{g}=g_0, \overline{g}=g_0+\epsilon h$. Throughout the rest of Section \ref{construct} the metric $g(t)$ will satisfy
$$|g(t)-\overline{g}|_\epsilon\leq C,$$
for some uniform constant $C$. This implies that $g(t)$ and $\overline{g}$ are equivalent uniformly in time. By the properties of $h$, we have that $\dot{g}$ and $\overline{g}$ are also equivalent uniformly in time. Therefore, when we prove $C^m$ and $H^m$ bounds with respect to one of $g, \dot{g}, \overline{g}$, $C^m$ and $H^m$ bounds with respect to the others follow trivially. 

We define

$$\delta(t)\coloneqq e^{\lambda t},$$

For a tensor $k$ on $M$, denote by $$||k ||^{(0)}_{H^j} =\left(\sum_i \int_M |\dot{\nabla}^i k|_0^2 dVol_{\dot{g}}\right)^{\frac{1}{2}},$$  $$||k ||^{(\epsilon)}_{H^j} =\left(\sum_i \int_M |\overline{\nabla}^i k|_\epsilon^2 dVol_{\overline{g}}\right)^{\frac{1}{2}}.$$

For a tensor on $M\times [a,b]$, denote by 
$$||k||^{(0)}_{H^j([a,b];H^k)}= \sum_{p\leq j, q\leq k}\left(\int_a^b\int_M |(\partial_t)^p\dot{\nabla}^{q}k|_0^2(t)dVol_{\dot{g}}dt\right)^\frac{1}{2},$$
$$||k||^{(\epsilon)}_{H^j([a,b];H^k)}= \sum_{p\leq j, q\leq k}\left(\int_a^b\int_M |(\partial_t)^p\overline{\nabla}^{q}k|_\epsilon^2(t)dVol_{\overline{g}}dt\right)^\frac{1}{2}.$$

Define
$$w\coloneqq v-(\delta(t)-\epsilon)h= g-g_0-\delta h.$$

We have the following,
\begin{align*}
	\partial_t (\delta(t)-\epsilon)h_{kk}&= \lambda \delta(t) h_{kk}\\
	&= (\delta(t)-\epsilon) \Delta^L h_{kk}-\lambda \epsilon h_{kk}\\
	&= \dot{g}^{\alpha \beta}\dot{\nabla}^2_{\alpha \beta} +2(\delta(t)-\epsilon)\dot{R}^{i\mbox{ }k}_{\mbox{ }i\mbox{ }k}h_{ii} \lambda\epsilon h_{kk}.
\end{align*}

Then, using (\ref{veq}), $w$ satisfies

\begin{equation}\label{weq}
	\partial_t w_{kk}= g^{\alpha\alpha}\overline{\nabla}_{\alpha}\overline{\nabla}_{\alpha}w_{kk}+A_{kk}+B_{kk}+D_{kk}-2(\delta(t)-\epsilon)2\dot{R}^{i\mbox{ }k}_{\mbox{ }i\mbox{ }k} h_{ii}-\lambda \epsilon h_{kk},
\end{equation}
where 
\begin{align*}
	A_{kk}&= (g^{\alpha\alpha}\overline{\nabla}_\alpha\overline{\nabla}_\alpha-\dot{g}^{\alpha\alpha}\dot{\nabla}_\alpha \dot{\nabla}_\alpha)((\delta(t)-\epsilon)h)_{kk},\\
	B_{kk}&= 2g^{ii}\overline{g}^{ii}(w_{ii}+(\delta(t)-\epsilon)h_{ii})(w_{kk}+(\delta(t)-\epsilon)h_{kk})\overline{g}^{kk}\overline{R}_{ikik}\\
	D_{kk}&=2g^{ii}\overline{g}^{ii}(w_{ii}+(\delta(t)-\epsilon)h_{ii})\overline{R}_{ikik}-2g_{kk}\overline{g}^{kk}\overline{R}_{kk}
	+ \delta_{rk}\frac{1}{2}(g^{ll})^2(\overline{\nabla}_{r}w_{ll}+\overline{\nabla}_{r}\delta(t)h_{ll})^2.
\end{align*}

Throughout Section \ref{construct}  it will be useful to have estimates of  $A,B,D$ under the assumption that  
$$\left|\overline{\nabla}^m \frac{d^p}{dt^p}w\right|_{\epsilon}(t)\leq C$$  for some $C>0$ and time $t$. In the following $V>0$ will be a constant depending on $C$.

Firstly, 
\begin{align*}
	A_{kk}&=(g^{\alpha\alpha}-\dot{g}^{\alpha\alpha})\overline{\nabla}_\alpha \overline{\nabla}_\alpha((\delta(t)-\epsilon)h)+ \dot{g}^{\alpha\alpha}(\overline{\nabla}_\alpha \overline{\nabla}_\alpha-\dot{\nabla}_\alpha \dot{\nabla}_\alpha)((\delta(t)-\epsilon)h)\\
	&=c^{ii}_{kk}w_{ii}+ (\delta^2(t)d+\epsilon \delta(t) e+ \epsilon^2 f),
\end{align*}
for some $c, d, e, f$ such that $c=O(\delta^2(t)),$ and 
\begin{align*}
	& \left|\left|\overline{\nabla}^m \frac{d^p}{dt^p} d \right|\right|_2^{(\epsilon)}, \left|\left|\overline{\nabla}^m \frac{d^p}{dt^p} e\right|\right|_2^{(\epsilon)} , \left|\left|\overline{\nabla}^m \frac{d^p}{dt^p} f \right|\right|_2^{(\epsilon)} \leq V \\
	&\left|\overline{\nabla}^m \frac{d^p}{dt^p} d \right|_\epsilon,\left|\overline{\nabla}^m \frac{d^p}{dt^p}e\right|_\epsilon,\left|\overline{\nabla}^m \frac{d^p}{dt^p}f\right|_\epsilon\leq V.
\end{align*}

We have
\begin{align*}
	B_{kk}&=2g^{ii}\overline{g}^{ii}(w_{ii}+(\delta(t)-\epsilon)h_{ii})(w_{kk}+(\delta(t)-\epsilon)h_{kk})\overline{g}^{kk}\overline{R}_{ikik}\\
	&= b^{ii}w_{ii}w_{kk}+c^{ii}_{kk}w_{ii}+\tilde{c}_{kk}w_{kk}+(d\delta^2)t+e\epsilon\delta(t)+f\epsilon)_{kk},
\end{align*}
for some $b,c,\tilde{c}, d, e, f$ such that
\begin{align*} &b^{ii}=2g^{ii}\overline{g}^{ii}\overline{g}^{kk}\overline{R}_{ikik}, \qquad \left|\overline{\nabla}^m \frac{d^p}{dt^p}b\right| \leq V,\\
	&c^{ii}_{kk}= 2g^{ii}\overline{g}^{ii}((\delta(t)-\epsilon))h_{kk}\overline{g}^{kk}\overline{R}_{ikik} \in H^p([t_0,t];H^m), \qquad \left|\overline{\nabla}^m \frac{d^p}{dt^p}c\right| \leq V \delta(t)\\
	&\tilde{c}_{kk}= 2g^{ii}\overline{g}^{ii}((\delta(t)-\epsilon))h_{ii}\overline{g}^{kk}\overline{R}_{ikik} \in H^p([t_0,t];H^m) \qquad \left|\overline{\nabla}^m \frac{d^p}{dt^p}c\right| \leq V\delta(t),\\
	& \left|\left|\overline{\nabla}^m \frac{d^p}{dt^p} d \right|\right|_2^{(\epsilon)}, \left|\left|\overline{\nabla}^m \frac{d^p}{dt^p} e\right|\right|_2^{(\epsilon)} , \left|\left|\overline{\nabla}^m \frac{d^p}{dt^p} f \right|\right|_2^{(\epsilon)} \leq V, \\
	&\left|\overline{\nabla}^m \frac{d^p}{dt^p} d \right|_\epsilon,\left|\overline{\nabla}^m \frac{d^p}{dt^p}e\right|_\epsilon,\left|\overline{\nabla}^m \frac{d^p}{dt^p}f\right|_\epsilon\leq V.\\
\end{align*}
Also,
\begin{align*}
	D_{kk}&=2g^{ii}\overline{g}^{ii}(w^{ii}+(\delta(t)-\epsilon)h_{ii})\overline{R}_{ikik}\\
	&= 2\overline{R}^{i\mbox{ }k}_{\mbox{ }i\mbox{ }k} w_{ii}+ c^{ii}_{kk}w_{ii}+(d\delta^2(t)+e\delta(t)\epsilon+f\epsilon^2)_{kk},
\end{align*}
for some $c, d, e, f$ such that
\begin{align*}
	&\left|\overline{\nabla}^m \frac{d^p}{dt^p}c\right|_\epsilon\leq V,\\
	& \left|\left|\overline{\nabla}^m \frac{d^p}{dt^p} d \right|\right|_2^{(\epsilon)}, \left|\left|\overline{\nabla}^m \frac{d^p}{dt^p} e\right|\right|_2^{(\epsilon)} , \left|\left|\overline{\nabla}^m \frac{d^p}{dt^p} f \right|\right|_2^{(\epsilon)} \leq V, \\
	&\left|\overline{\nabla}^m \frac{d^p}{dt^p} d \right|_\epsilon,\left|\overline{\nabla}^m \frac{d^p}{dt^p}e\right|_\epsilon,\left|\overline{\nabla}^m \frac{d^p}{dt^p}f\right|_\epsilon\leq V.
\end{align*}

Furthermore,
$$2g_{kk}\overline{g}_{kk}\overline{R}_{kk}= 2\overline{R}_{kk}=\overline{R}_{kk}\overline{g}^{kk}w_{kk}+(\delta(t)-\epsilon)h_{kk}\overline{g}^{kk}\overline{R}_{kk}= cw_{kk} +\epsilon \overline{f}_{kk},$$ for some $c, \overline{f}$ such that 
\begin{align*}
	&\left|\overline{\nabla}^m \frac{d^p}{dt^p}c\right|_\epsilon\leq V \delta(t),\\ 
	& \overline{f}\in H^m  \text{ is independent of time}.
\end{align*}
Also,
$$\delta_{kr}\frac{1}{2}(g^{ll})^2(\overline{\nabla}_{r}w_{ll}+\overline{\nabla}_{r}\delta(t)h_{ll})^2
=\overline{b}^{ll}(\overline{\nabla}_{r}w_{ll})^2+\overline{c}^{ll}_r\overline{\nabla}_r
w_{ll}+\delta^2(t)d_{rr},$$
for some $\overline{b}, \overline{c}$ such that
$$\left|\overline{\nabla}^m \frac{d^p}{dt^p}\overline{b}\right|_\epsilon\leq V, \left|\overline{\nabla}^m \frac{d^p}{dt^p}\overline{c}\right|_\epsilon\leq V \delta(t).$$

Therefore, using Lemma \ref{vdecay},
\begin{equation}
	\partial_t w= E(w)+F, \qquad w=0 \text{ on } M\times \{t_0\}
\end{equation}
where
\begin{align*}
	E(w)_{kk}&=g^{\alpha\alpha}\overline{\nabla}_\alpha\overline{\nabla}_\alpha w_{kk}+2\overline{R}^{i\mbox{ }k}_{\mbox{ }i\mbox{ }k} w_{ii}+\delta_{rk}((\overline{b}^{ll}\overline{\nabla}_rw_{ll})^2+\overline{c}^{ll}_r\overline{\nabla}_r w_{ll})\\
	&\hspace{4mm}+b^{ii}w_{ii}w_{kk}+c^{ii}_{kk}w_{ii}+ \tilde{c}_{kk}w_{kk},\\
	F&=d\delta^2+ e\epsilon\delta(t)+f\epsilon^2+\overline{f}\epsilon,
\end{align*}
for some two tensors $b, \overline{b}, c, \overline{c}, \tilde{c}, d, e, f, \overline{f}$ satisfying, under the  assumption that  
$$\left|\overline{\nabla}^m \frac{d^p}{dt^p}w\right|_{\epsilon}\leq C,$$ the following bounds,
\begin{align*}
	&\left|\overline{\nabla}^m \frac{d^p}{dt^p}\overline{b}\right|_\epsilon, \left|\overline{\nabla}^m \frac{d^p}{dt^p}b\right|_\epsilon \leq V,\\
	&\left|\overline{\nabla}^m c\right|_\epsilon,  \left|\overline{\nabla}^m \frac{d^p}{dt^p}c\right|_\epsilon, \left|\overline{\nabla}^m \frac{d^p}{dt^p}c\right|_\epsilon \leq V \delta(t).\\
	& \left|\left|\overline{\nabla}^m \frac{d^p}{dt^p} d \right|\right|_2^{(\epsilon)}, \left|\left|\overline{\nabla}^m \frac{d^p}{dt^p} e\right|\right|_2^{(\epsilon)} , \left|\left|\overline{\nabla}^m \frac{d^p}{dt^p} f \right|\right|_2^{(\epsilon)} \leq V, \\
	&\left|\overline{\nabla}^m \frac{d^p}{dt^p} d \right|_\epsilon,\left|\overline{\nabla}^m \frac{d^p}{dt^p}e\right|_\epsilon,\left|\overline{\nabla}^m \frac{d^p}{dt^p}f\right|_\epsilon\leq V.
\end{align*}
for some $V=V(C)>0.$

\subsection{$C^m$ bounds}\label{C^mbounds}
Theorem \ref{l^2long} will be the main technical result used in the proof of Theorem \ref{mainthm}. Proving Theorem \ref{l^2long} will require $C^m$ bounds on $w$ and its derivates in time and space, so that we can control certain terms not accounted for by an appropriate linear operator. The goal of Section \ref{C^mbounds} will be to prove the following proposition. 
\begin{proposition}\label{lastC^m}
	Suppose $w(t)$ exists for $t\in[t_0, t_0+kT]$.
	There exists $M_j,T>0$ and $N\leq 0$ such that if 
	$$\left|\left|\frac{d^j}{dt^j} \overline{\nabla}^iw\right|\right|_2^{\epsilon}(t) \leq M_{i+j}\delta(t), \qquad i+j\leq 7,$$
	for all $t\in [t_0, t_0+(k-1)T]$ and $t_0+kT<N$, 
	then for all $m,p$
	$$\left|\frac{d^p}{dt^p} \overline{\nabla}^m w\right|_\epsilon\leq C_{m,p} \delta(t),$$
	for all $t\in [t_0, t_0+kT].$
\end{proposition}

The proof if Proposition \ref{lastC^m} will be broken up into proving the the next few lemmas. The following lemma shows that having $H^7$ bounds on $w$  means we have $C^m$ bounds for $w$ independent of $t_0,k$. The reason for number 7 here is that it gives us, by Lemma \ref{sobelevineq}, $C^4$ bounds which will be needed to extract a suitable convergent subsequence in the proof of Theorem \ref{anicentRDTF}.

\begin{lemma}\label{longtimeests}
	Suppose $N<0$ is sufficiently negative and $t_0+(k-1)T<N$.
	If 
	$$|| \overline{\nabla}^iw||_2^{(\epsilon)}(t) \leq M_{i}\delta(t), \qquad i\leq 7,$$
	for all $t\in [t_0, t_0+(k-1)T]$, then 
	there is a $C_i=C_i(N)>0$ such that
	$$|\overline{\nabla}^i w|_\epsilon(t) \leq C_i,$$
	for all $t\in [t_0, t_0+(k-1)T]$, where $C_i(N) \rightarrow 0$ as $N \rightarrow -\infty$.
\end{lemma}			

\begin{proof}
	The limit $\delta(t)\rightarrow 0$ as $t\rightarrow -\infty$ and
	Lemma \ref{sobelevineq} gives 
	$$||w||_{C^3}(t) \leq C=C(N),$$
	for all $t\in [t_0, t_0+(k-1)T]$, where $C$ can be made as small as we like by making $N$ sufficiently negative. Suppose 
	$$|\overline{\nabla}^m w|_\epsilon(t) \leq C_m(C),$$
	for $m\geq 3$ and $t\in [t_0, t_0+(k-1)T]$, with $C_m(C) \rightarrow 0$ as $C\rightarrow 0$.  Note that we have the commutation relation for an arbitrary tensors $T$,
	$$[\overline{\Delta}, \overline{\nabla}]T= \overline{\nabla}\overline{R} * T+ \overline{R}*\overline{\nabla}T.$$
	Therefore,
	\begin{align*}
		\overline{\nabla}(g^{\alpha\alpha}\overline{\nabla}^2_{\alpha\alpha})&= \overline{\nabla}(\overline{g}^{\alpha\alpha}\overline{\nabla}^2_{\alpha\alpha}w+ (g^{\alpha\alpha}-\overline{g}^{\alpha\alpha})\overline{\nabla}^2_{\alpha\alpha}w)\\
		&= \overline{\Delta}(\overline{\nabla}w)+ (g^{\alpha\alpha}-\overline{g}^{\alpha\alpha})\overline{\nabla}\overline{\nabla}^2_{\alpha\alpha}w+ \overline{\nabla}(g^{\alpha\alpha}-\overline{g}^{\alpha\alpha})\overline{\nabla}^2_{\alpha\alpha}w+ \overline{\nabla}R*w+R*\overline{\nabla}w\\
		&= \overline{\Delta}(\overline{\nabla}w)+ (g^{\alpha\alpha}-\overline{g}^{\alpha\alpha})\overline{\nabla}\overline{\nabla}^2_{\alpha\alpha}w+ \overline{\nabla}(g^{\alpha\alpha}-\overline{g}^{\alpha\alpha})\overline{\nabla}^2_{\alpha\alpha}w+ \overline{\nabla}R*w+R*\overline{\nabla}w\\
		&= L(\overline{\nabla} w)+  \overline{\nabla}R*w+R*\overline{\nabla}w.
	\end{align*}
	for some uniformly elliptic in space and time linear  operator $L$ with uniformly bounded coefficients by making $C_0(N), C_1(N)$ sufficient small by making $N$ sufficiently negative. 
	Thus, the time-evolution equation for $\overline{\nabla} w$ is
	\begin{equation*}
		\partial_{t} (\overline{\nabla} w)= L(\overline{\nabla}^m w)+ F(1,w,...,\overline{\nabla}^{2}w),
	\end{equation*}
	where $F$ is linear function in $1,w,...,\overline{\nabla}^{2}w$ with bounded (in time and space) coefficients (i.e a homogenous linear with bounded coefficients).
	Similarly, for $C_m$ small enough, the time-evolution equation for $\overline{\nabla}^m w$ is
	\begin{equation}\label{nable}
		\partial_{t} (\overline{\nabla}^m w)= L_m (\overline{\nabla}^m w)+ F_m(1,w,...,\overline{\nabla}^{m+1}w),
	\end{equation}
	where $L_m$ is uniformly elliptic and $F_m$ is linear function in $1,w,...,\overline{\nabla}^{m+1}w$ with bounded (in time and space) coefficients. It is crucial that $m\geq 1$ as the evolution equation for $w$ contains a term quadratic in $\overline{\nabla}w$.

	Thus, by \cite[Section 7, Thm 6.1]{Lq}, tells us, since $w(t_0, \cdot)=0$ and $w(.,\infty)=0$, that 
	$$|\overline{\nabla}^{m+1} w|_\epsilon\leq C_{m+1}(C),$$
	for all $t\in [t_0, t_0+(k-1)T]$ and where $C_{m+1}(C)\rightarrow 0$ as $C\rightarrow 0$.
	Therefore, by induction, the statement of the lemma follows.
\end{proof}

The following lemma says that if we can bound the $C^m$ bounds on $w(t)$ for $t\in [t_0, t_0+(k-1)T]$, then we can bound it for another time $T$, provided $T$ is small enough.
\begin{lemma}\label{shortC^mests}
	Suppose that $w(t)$ exists for $t\in[t_0, t_0+(k-1)T]$.
	Let $i\geq 3$. If
	$$||w||^{(\epsilon)}_{C^i}(t) \leq C,$$
	for all $t\in [t_0, t_0+(k-1)T]$ and $C, T$ small enough, then $w$ exists for another time $T$ and 
	$$||w||^{(\epsilon)}_{C^i}(t) \leq C_i(C),$$
	for all $t\in [t_0+(k-1)T, t_0+kT]$, where $C_i(C)\rightarrow 0$ as $C\rightarrow 0$.
\end{lemma}

\begin{proof}
	For $C>0$ sufficiently small, the $C^3$ bounds mean, by Lemma \ref{shorttimeests}, $w(t)$ exists for another time $T$ we have a bound on $|w|_\epsilon$ and $|\overline{\nabla} w|_\epsilon$ on $[t_0+(k-1)T, t_0+kT]$ for $T$ small enough. Following the proof of Lemma \ref{longtimeests} proves  the lemma: once we have bounds 
	$$|\overline{\nabla}^iw|_{\epsilon}\leq C,$$
	for $i\leq m$, the equation for $\overline{\nabla}^m w$ for $m>0$ is 
	\begin{equation}
		\partial_{t} (\overline{\nabla}^m w)= L (\overline{\nabla}^m w)+ F_m\left(1,w,...,\overline{\nabla}^{m+1}w\right),
	\end{equation}
	where $L$ is uniformly elliptic and $F$ is linear function in $1,w,...,\overline{\nabla}^{m+1}w$ with bounded coefficients  on $M\times [t_0, t_0+kT]$, assuming $C$ is small enough. Thus, $$|\overline{\nabla}^{m+1}w|_{\epsilon}\leq C_{m+1},$$ on  $M\times [t_0, t_0+kT]$.
\end{proof}		
The following lemma extends the bounds of Lemmas \ref{longtimeests}, \ref{shortC^mests} to derivatives in time.
\begin{lemma}\label{i,m}
	Suppose $N<0$ is sufficiently negative and $t_0+kT<N$.
	Suppose $w(t)$ exists for $t\in[t_0, t_0+kT]$.
	If 
	$$\left|\left| \overline{\nabla}^i\frac{d^j w }{dt^j}\right|\right|_2^{\epsilon}(t) \leq M_{i,j}\delta(t), \qquad i+j\leq 7,$$
	for all $t\in [t_0, t_0+(k-1)T]$, then  for all $q,m$ 
	there is a $C_{q,m}=C_{q,m}(N)>0$ such that
	$$\left|\frac{d^q \overline{\nabla}^m w}{dt^q}\right|_\epsilon(t) \leq C_{q,m},$$
	for all $t\in [t_0, t_0+kT]$, where $C_{q,m}(N)\rightarrow 0$ as $N\rightarrow -\infty$.
\end{lemma}	

\begin{proof}
	We first note that if $N$ is sufficiently negative (depending on the $M_{i,j}$), then $M_{i,j}\delta(t)\leq 1$ for all $t\in [t_0, t_0+kT]$.  
	The equation (\ref{weq}) and the bounds on the derivatives of $w$ with respect to space imply we have a bound on $|\partial_{t} w|_{\epsilon}$. Similarly, we can write the evolution equation for 
	$\frac{d^q \overline{\nabla}^m w}{dt^q}$ as		
	$$\partial_t \left(\frac{d^q \nabla^m}{dt^q}w\right)= F_{q,m}(1, \overline{\nabla}w, ...,  \overline{\nabla}^p w),$$
	for $p$ less than some finite number. Consequently, by Lemma \ref{longtimeests} and \ref{shortC^mests}, $$\left|\frac{d^q \nabla^m w}{dt^q}\right|_{\epsilon} \leq C_{q,m}$$.
\end{proof}

\begin{proof}(of Proposition \ref{lastC^m})
	Lemma \ref{sobelevineq} gives 
	$$||w||_{C^3}^\epsilon(t) \leq C\delta(t),$$
	for all $t\in [t_0, t_0+(k-1)T]$, where $C=C(M_j)>0$. If $N$ is sufficiently negative, $|Rm(g(t))|_{g(t)}\leq K$, and so by Lemma \ref{shorttimeests}, $$|w-w(t_0+(k-1))|_{g(t_0+(k-1))}, |\overline{\nabla} w|_{g(t_0+(k-1))}(t) \leq \epsilon \leq \delta(t),$$ for $t\in [t_0+(k-1)T, t_0+kT]$. Thus, there is $C'(N)>0$ such that
	$$|w|_\epsilon, |\overline{\nabla} w|_\epsilon \leq C'\delta(t),$$
	for $t\in [t_0+(k-1)T, t_0+kT].$

	Suppose that
	$$||w||_{C^m}^\epsilon(t) \leq C_m(C),$$
	for $m\geq 3$ and $t\in [t_0, t_0+(k-1)T]$, with $C_m \rightarrow 0$ as $C\rightarrow 0$. By Lemma \ref{i,m} we have,
	\begin{equation}
		\partial_{t} (\overline{\nabla}^m w)= L_m (\overline{\nabla}^m w)+ F_m(1,w,...,\overline{\nabla}^{m+1}w),
	\end{equation}
	where $F$ is linear function in $1,w,...,\overline{\nabla}^{m+1}w$ with bounded coefficients (in time and space). Furthermore, the norm coefficient of the term  $\overline{\nabla}^{m+1}w$ is less or equal to $C\delta(t)$ for some $C=C(C_i;i\leq m)$.
	The proof of \cite[Chapter 7, Thm 6.1]{Lq} tells us, since $w(t_0, \cdot)=0$ and $w(\cdot,\infty)=0$, that 
	$$\left|\overline{\nabla}^{m+1} w\right|\leq C_{m+1}(C)\delta(t),$$
	for all $t\in [t_0, t_0+(k-1)T]$.
	Therefore, by induction, $|\overline{\nabla}^{m}w|_\epsilon \leq C_{m,0}\delta(t)$ for all $m$. The time derivatives bounds can be obtained by writing them in terms of space derivatives. 
\end{proof}

\subsection{$H^m$ bounds}\label{L^2bounds}

Suppose $w(t)$ for $t\in[t_0, t_0+kT]$.
Using the method of Proposition \ref{vdecay}, the bounds of Propositions \ref{boundsonh}, \ref{lastC^m}, it follows that if 
$$\left|\left|\frac{d^j}{dt^j} \overline{\nabla}^iw\right|\right|_2^{\epsilon}(t) \leq M_{i+j}\delta(t), \qquad i,j\leq 7,$$
for $t\in[t_0, t_0+(k-1)T]$, 
then 
\begin{equation}\label{wrbound}
	\left|\frac{d^p}{dt^p} \overline{\nabla}^m w\right|_\epsilon(r)\leq C_{m,p,l} \delta(t)r^{-l},
\end{equation}
for all $l$.
Therefore
$$\left|\left|\frac{d^j}{dt^j} \overline{\nabla}^iw\right|\right|_2^{\epsilon}(t),$$
is finite for all $t\in[t_0+(k-1)T, t_0+kT]$.

\begin{theorem}\label{l^2long}
	Suppose $w(t)$ exists for $t\in[t_0, t_0+kT]$.
	There exists $M_{i,j},T>0$ and $N\leq 0$ such that if 
	$$\left|\left|\frac{d^j}{dt^j} \overline{\nabla}^iw\right|\right|_2^{\epsilon}(t) \leq M_{i,j}\delta(t), \qquad i+j\leq 7,$$
	for all $t\in [t_0, t_0+(k-1)T]$ and $t_0+kT<N$, then 
	$$\left|\left|\frac{d^j}{dt^j} \overline{\nabla}^iw\right|\right|_2^{\epsilon}(t) \leq M_{i,j}\delta(t),\qquad i+j\leq 7,$$
	for all $t\in[t_0, t_0+kT]$. 
\end{theorem}

\begin{proof}
	By Proposition \ref{lastC^m} and the analysis of \ref{groundwork}, for $t\in[t_0, t_0+kT]$ we have the following bounds:
	\begin{align*}
		&\left|\overline{\nabla}^i \partial_{t}^p w\right|_{\epsilon}\leq C_{i,p}\delta(t),\\
		&\left|\frac{d^i\nabla^m}{dt^i} b\right|_\epsilon \leq V\\
		&\left|\frac{d^i\nabla^m}{dt^i} \overline{c}\right|_\epsilon, \left|\frac{d^i\nabla^m}{dt^i} c\right|_\epsilon, \left|\frac{d^i\nabla^m}{dt^i} \tilde{c}\right|_\epsilon \leq C_{i,m}\delta(t),\\
		& \left|\left|\frac{d^i \nabla^m}{dt^i} d\right|\right|_2^\epsilon, \left|\left|\frac{d^i \nabla^m}{dt^i} e\right|\right|_2^\epsilon, 
		\left|\left|\frac{d^i \nabla^m}{dt^i} f\right|\right|_2^\epsilon, \left|\left|\frac{d^i \nabla^m}{dt^i} \overline{f}\right|\right|_2^\epsilon\leq V,\\
		&\left|\frac{d^i \nabla^m}{dt^i} d\right|\epsilon, \left|\frac{d^i \nabla^m}{dt^i} e\right|\epsilon, \left|\frac{d^i \nabla^m}{dt^i} f\right|_\epsilon, \left|\frac{d^i \nabla^m}{dt^i} \overline{f}\right|_\epsilon \leq V.
	\end{align*}
	
	Also $\overline{f}$ is time-independent.
	We will only need the bounds $$\left|\overline{\nabla}^i \partial_{t}^p w\right|_{\epsilon}\leq C_{i,p}\delta(t), \left|\frac{d^i\nabla^m}{dt^i} \overline{c}\right|_\epsilon, \left|\frac{d^i\nabla^m}{dt^i} c\right|_\epsilon, \left|\frac{d^i\nabla^m}{dt^i} \tilde{c}\right|_\epsilon \leq C_{i,m}\delta(t),$$
	for indices up to a finite number. So can assume $C_{i,p}, C_{i,m} \leq V$.
	
	These bounds imply that all the inhomogenous linear equations arising from the time-evolution equations of $w$ and its derivatives in time and space will have bounded coefficients (uniformly in choice of $t_0$ and $N$).
	
	We will prove the following bounds:
	\begin{align*}
		&(||w_{t^p}||_2^{(\epsilon)})^2\leq C(M_{j,m}, V, \lambda)\delta^3(t)+ C(V, \lambda)\left(1+\sum_{j\leq 1, m\leq p\leq 7}M_{j,m}\right),\\
		&(||\overline{\nabla}^i w_{t^p}||_2^{(\epsilon)})^2 \leq C(M_{j,m},V,\lambda)\delta^3(t)+C(M_{j,m},V,\lambda)T\delta^2(t)+Q(V, \lambda, M_{j,m} ; j<i, m\leq p),
	\end{align*}
	for $i\geq 1$ and 
	for $Q$ a quadratic polynomial in $V, \lambda, M_{j,m}$ for $j<i, m\leq p$.
	So we will have proven the theorem if first we choose the $M_{j,p}$ so that 
	\begin{align}\label{Mbound}
		&2C(V, \lambda)\left(1+\sum_{j\leq 1, m\leq p\leq 7} M_{j,m}\right)\leq M^2_{0,p} \qquad p\leq 7,\\
		&2Q(V, \lambda, M_{j,m} ; j<i, m\leq p) \leq M^2_{i,p},
	\end{align}
	for $i\geq 1, i+p\leq 7.$ This system of inequalities is satisfied for large enough $M_{j,m}$.
	Then choose $T$ to be small enough and $N$ negative enough.
	
	First we will prove the bound on $||w||_2^{(\epsilon)}.$ Then we prove the bound on $||\partial_{t} w||_2^{(\epsilon)}$, and explain how to bound $||\partial_{t}^p w||_2^{(\epsilon)}$. Then the bound on $||\overline{\nabla} w||_2^{(\epsilon)}$, then explain how to bound the higher derivatives $||\overline{\nabla}^m \partial_{t}^p w||_2^{(\epsilon)}$.

	Recall the time evolution equation for $w$:
	\begin{equation*}
		\partial_t w= E(w)+F, \qquad w=0 \text{ on } M\times \{t_0\}, 
	\end{equation*}
	where
	\begin{align*}
		&E(w)_{kk}=g^{\alpha\alpha}\overline{\nabla}_\alpha\overline{\nabla}_\alpha w_{kk}+2\overline{R}^{i\mbox{ }k}_{\mbox{ }i\mbox{ }k} w_{ii}+\delta_{rk}((\overline{b}^{ll}\overline{\nabla}_rw_{ll})^2+\overline{c}^{ll}_r\overline{\nabla}_r w_{ll})\\
		&\hspace{4mm}+b^{ii}w_{ii}w_{kk}+c^{ii}_{kk}w_{ii}+ \tilde{c}_{kk}w_{kk},\\
		&F=d\delta^2+ e\epsilon\delta(t)+f\epsilon^2+\overline{f}\epsilon.
	\end{align*}
	Using the bounds stated at the start of the proof of Theorem \ref{l^2long}, we have the following bounds for $t\in [t_0, t_0+(k-1)T]$,
	\begin{align*}
		&||\delta_{rk}((\overline{b}^{ll}\overline{\nabla}_rw_{ll})^2+\overline{c}^{ll}_r\overline{\nabla}_r w_{ll})+b^{ii}w_{ii}w_{kk}+c^{ii}_{kk}w_{ii}+ \tilde{c}_{kk}w_{kk}||_2^{(\epsilon)}\leq C(M_{j,m}, V)\delta^2(t)\\
		&||\delta^2(t)d+ \epsilon\delta(t)e+\epsilon^2 f||_2^{(\epsilon)}\leq C(V)\delta^2(t),\\
		&||\epsilon\overline{f}||_2^{(\epsilon)}\leq \epsilon C(V).
	\end{align*}
	By (\ref{wrbound}), for $t\in [t_0, t_0+kT]$,
	\begin{align*}
		&\int_M \langle g^{\alpha\alpha}\overline{\nabla}_{\alpha}\overline{\nabla}_{\alpha}w, w\rangle +\int_M 2\overline{R}^{i\mbox{ }k}_{\mbox{ }i\mbox{ }k} w_{ii}w_{kk}\\
		&\leq \int_M \langle \dot{\Delta} w ,w \rangle_0+ 2\dot{R}^{i\mbox{ }k}_{\mbox{ }i\mbox{ }k} w_{ii}w_{kk}+ C(M_{j,m}, V)\delta^3(t).
	\end{align*}

	Thus, by definition of $\lambda$, we have
	\begin{align*}
		\partial_t (||w||_2^{(\epsilon)})^2&= \int_M \langle \partial_t w, w \rangle\\
		& \leq \int_M \langle  g^{\alpha\alpha}\overline{\nabla}_{\alpha}\overline{\nabla}_{\alpha}w, w\rangle +\int_M 2\overline{R}^{i\mbox{ }k}_{\mbox{ }i\mbox{ }k} w_{ii}w_{kk}+ C(M_{j,m}, V)\delta^3(t)+ \int_M \langle \epsilon\overline{f},w \rangle,\\
		&\leq \int_M \langle \Delta_0 w ,w \rangle_0+ 2\dot{R}^{i\mbox{ }k}_{\mbox{ }i\mbox{ }k} w_{ii}w_{kk}+ C(M_{j,m}, V)\delta^3(t)+\epsilon V M_{1,0} \delta(t),\\
		&\leq \int_M -|\dot{\nabla}w|^2_0+ 2\dot{R}^{i\mbox{ }k}_{\mbox{ }i\mbox{ }k} w_{ii}w_{kk}+ C(M_{j,m}, V)\delta^3(t)+\epsilon V M_{1,0} \delta(t),\\
		&\leq 2\lambda(||w||_2^{(\epsilon)})^2+ C(M_{j,m},V)\delta^3(t)+\epsilon V M_{1,0} \delta(t).
	\end{align*}
	The uniform bounds on the norms of $\overline{\nabla}w_{t^{p}}$ mean, by the heat estimates of \cite[Chapter 7, Thm 6]{Evans}, there is a constant $C_0$ such that
	\begin{align*}
		||w||_2^{(\epsilon)}(t)+||\overline{\nabla}w||_2^{(\epsilon)}(t) &\leq C_0(||F||_{L^2([t_0, t_0+(k-1)T])}^{(\epsilon)}+ ||w||_{H^1}^{(\epsilon)}(t_0+(k-1)T))\\
		&\leq C_0(4VT^{\frac{1}{2}}+M_{1,0}+M_{2,0}))\delta(t).
	\end{align*}
	
	This bound implies that on $t\in [t_0+(k-1)T, t_0+kT]$ we have,
	\begin{align*}
		&||\delta_{rk}((\overline{b}^{ll}\overline{\nabla}_rw_{ll})^2+\overline{c}^{ll}_r\overline{\nabla}_r w_{ll})+b^{ii}w_{ii}w_{kk}+c^{ii}_{kk}w_{ii}+ \tilde{c}_{kk}w_{kk}||_2^{(\epsilon)}\leq C(M_{j,m}, V)\delta^2(t),\\
		&||\delta^2(t)d+ \epsilon\delta(t)e+\epsilon^2 f||_2^{(\epsilon)}\leq C(V)\delta^2(t),\\
		&||\epsilon\overline{f}||_2^{(\epsilon)}\leq \epsilon C(V).
	\end{align*}
	
	So
	$$\partial_t (||w||_2^{(\epsilon)})^2\leq 2\lambda (||w||_2^{(\epsilon)})^2+ C(M_{j,m}, V)\delta^3(t)+\epsilon C_0(4V+M_{1,0}+M_{2,0})V\delta(t).$$
	
	Since $w(t_0)=0$, 	Gronwall's inequality implies that
	\begin{align*}
		(||w||_2^{(\epsilon)})^2 &\leq \delta^2(t)\left(\int_{t_0}^{t}C(M_{j,m},V)\delta^3(s)\delta^{-2}(s) ds+ \int_{t_0}^{t}\epsilon C_0(4V+M_{1,0}+M_{2,0})V\delta(s)\delta^{-2}(s) ds\right)\\
		&\leq C(M_{j,m}, V, \lambda)\delta^3(t)+\frac{\epsilon C_0}{\lambda}(4V+M_{1,0}+M_{2,0})V \delta^2(t)\delta^{-1}(t_0)\\
		&\leq C(M_{j,m}, V, \lambda)\delta^3(t)+\frac{C_0}{\lambda}(4V+M_{1,0}+M_{2,0})V \delta^2(t).
	\end{align*}
	with the last inequality due to $\epsilon= e^{\lambda t_0}.$ This proves the bound (\ref{Mbound}) for $w$.
	
	$\partial_{t}^p w$ for $p\leq 7$: The time-evolution equation for $\partial_{t}^p w$ is:
	\begin{align*}
		\partial_t (\partial_{t}^p w)_{kk}&= \overline{\Delta} (\partial_{t}^p w)+ \partial_t^{p}((w+\epsilon h)*\overline{\nabla}^2 w)+
		2\overline{R}^{i\mbox{ }k}_{\mbox{ }i\mbox{ }k} \partial_{t^{p}}(w_{ii})\\
		&\hspace{5mm}+ \partial_{t^{p}}(\delta_{rk}((\overline{b}^{ll}\overline{\nabla}_rw_{ll})^2+\overline{c}^{ll}_r\overline{\nabla}_r w_{ll})
		+b^{ii}w_{ii}w_{kk}+c^{ii}_{kk}w_{ii}+ \tilde{c}_{kk}w_{kk}) +\partial_t^{p}F.
	\end{align*}
	
	Using the bounds stated at the start of Theorem \ref{l^2long}, we have the following bounds on $t\in [t_0, t_0+(k-1)T]$,
	\begin{align*}
		&||\partial_t^{p}((w+\epsilon h)*\overline{\nabla}^2 w)||_2^{(\epsilon)} \leq C(M_{j,m}, V)\delta^2(t),\\
		&||\partial_t^{p}(\delta_{rk}((\overline{b}^{ll}\overline{\nabla}_rw_{ll})^2+\overline{c}^{ll}_r\overline{\nabla}_r w_{ll})+b^{ii}w_{ii}w_{kk}+c^{ii}_{kk}w_{ii}+ \tilde{c}_{kk}w_{kk})||_2^{(\epsilon)}\leq C(M_{j,m}, V)\delta^2(t),\\
		&||\partial_t^{p}(\delta^2(t)d+ \epsilon\delta(t)e+\epsilon^2 f)||_2^{(\epsilon)}\leq C(V)\delta^2(t).
	\end{align*}
	
	Thus,
	\begin{align*}
		\partial_t (||w_t^{p}||_2^{(\epsilon)})^2&
		= 2\int_M \langle\partial_t(w_t^{p}), \partial_{t}^p w \rangle_\epsilon \\
		&\leq \int_M \langle \dot{\Delta} \partial_{t}^p w ,\partial_{t}^p w \rangle_0+ 2\dot{R}^{i\mbox{ }k}_{\mbox{ }i\mbox{ }k} (\partial_{t}^p w)_{ii}(\partial_{t}^p w)_{kk}+ C(M_{j,m}, V)\delta^3(t)+\epsilon V M_{1,0} \delta(t),\\
		&\leq \int_{M} -|\dot{\nabla}(\partial_{t}^p w)|^2+  2\dot{R}^{i\mbox{ }k}_{\mbox{ }i\mbox{ }k} (\partial_{t}^p w)_{ii}(\partial_{t}^p w)_{kk}+ C(M_{j,m}, V)\delta^3(t)+\epsilon C(V, \lambda)M_{0,p}\delta(t),\\
		&\leq 2\lambda(||w_{t^{p}}||_2^{(\epsilon)})^2+ C(M_j, V)\delta^3(t)+\epsilon C(V, \lambda)M_{0,p}\delta(t).
	\end{align*}
	
	The heat estimates of \cite[Chapter 7, Thm 6]{Evans} give that on $t\in[t_0+(k-1)T, t_0+kT]$,
	$$||\partial_{t}^p w||_2^{(\epsilon)}(t)+ ||\overline{\nabla}(\partial_{t}^p w)||_2^{(\epsilon)}(t)\leq C_0(C(V)T+M_{0,p}+M_{1,p})\delta(t).$$
	
	Consequently, on $t\in[t_0+(k-1)T, t_0+kT]$, 
	$$\partial_t (||w_{t^{p}}||_2^{(\epsilon)})^2
	\leq 2\lambda(||\partial_{t}^p w||_2^{(\epsilon)})^2+ C(M_j, V)\delta^3(t)+\epsilon C(V, \lambda)(1+M_{0,p}+M_{1,p})\delta(t).$$
	
	Gronwall's inequality implies that
	\begin{align*}
		(||w||_2^{(\epsilon)})^2 &\leq \delta^2(t)\left(\int_{t_0}^{t}C(M_j,V)\delta^3(s)\delta^{-2}(s) ds+ \int_{t_0}^{t}\epsilon C(V, \lambda)(1+M_{0,p}+M_{1,p})\delta(s)\delta^{-2}(s) ds\right)\\
		&\leq C(M_j, V, \lambda)\delta^3(t)+\frac{\epsilon C(V, \lambda)(1+M_{0,p}+M_{1,p})}{\lambda} \delta^2(t)\delta^{-1}(t_0)\\
		&\leq C(M_j, V, \lambda)\delta^3(t)+\frac{C(V, \lambda)(1+M_{0,p}+M_{1,p})}{\lambda} \delta^2(t),
	\end{align*}
	with the last inequality due to $\epsilon= e^{\lambda t_0}.$	
	
	$\overline{\nabla}w:$ Using Young's inequality, on $t\in[t_0, t_0+(k-1)T]$, we have the following bounds
	\begin{align*}
		\int_M \langle \delta^2(t)d+\epsilon\delta(t)e+\epsilon^2 f+\epsilon \overline{f}, \partial_{t} w \rangle &\leq 
		\frac{1}{4}(||\partial_{t} w||_{2}^{(\epsilon)})^2+ (||\delta^2(t)d+\epsilon\delta(t)e+\epsilon^2 f+\epsilon \overline{f}||_2^{(\epsilon)})^2\\
		&\leq \frac{1}{4}(||\partial_{t} w||_{2}^{(\epsilon)})^2+ C(V)\delta^2(t).
	\end{align*}
	Similarly,
	\begin{align*}
		\int_M \langle \delta_{rk}((\overline{b}^{ll}\overline{\nabla}_rw_{ll})^2+\overline{c}^{ll}_r\overline{\nabla}_r w_{ll})+b^{ii}w_{ii}w_{kk}+c^{ii}_{kk}w_{ii}+ \tilde{c}_{kk}w_{kk}, w_t \rangle \leq \frac{1}{4}(||\partial_{t} w||_{2}^{(\epsilon)})^2+ C(V)\delta^4(t).
	\end{align*}
	Thus, for $t\in[t_0, t_0+(k-1)T]$, 
	\begin{align*}
		||\partial_{t} w||_2^{(\epsilon)}(t)&= \int_M \langle \partial_{t} w, \partial_{t} w \rangle_\epsilon,\\
		&=\int_M \langle g^{\alpha\alpha}\overline{\nabla}_\alpha\overline{\nabla}_\alpha w, \partial_{t} w \rangle_\epsilon+2\int_M  2\overline{R}^{i\mbox{ }k}_{\mbox{ }i\mbox{ }k} (\partial_{t} w)_{ii}(\partial_{t} w)_{kk} + C(M_{j,m}, V)\delta^4(t)\\
		&\hspace{4mm}+\frac{1}{2}(||\partial_{t} w||_2^{(\epsilon)})^2(t)+ C(V,\lambda)\delta^2(t).
	\end{align*}
	Using (\ref{wrbound}),
	$$ \langle g^{\alpha\alpha}\overline{\nabla}_\alpha\overline{\nabla}_\alpha w, \partial_{t} w \rangle_\epsilon \leq   \int_M \langle \overline{\Delta} w, \partial_{t} w \rangle_\epsilon+ C(M_{j,m})\delta^3(t),$$ and 
	\begin{align*}
		\int_{t_0}^{t_0+(k-1)T} \int_M 2\overline{R}^{i\mbox{ }k}_{\mbox{ }i\mbox{ }k} w_{ii}(\partial_{t} w)_{kk}
		&\leq \int_{t_0}^{t_0+(k-1)T} \frac{1}{4}(||\partial_{t} w||_2^{(\epsilon)})^2+ C(||w||_{2}^{(\epsilon)})^2\\
		&\leq \frac{1}{4}(||\partial_{t} w||_2^{(\epsilon)})^2 +CM_{1,0}^2 \delta^2(t),
	\end{align*}
	we have
	\begin{align*}
		||\partial_{t} w||_2^{(\epsilon)}(t)
		&\leq -\int_M \langle \overline{\nabla}  w, \overline{\nabla}\partial_{t} w \rangle_\epsilon+\int_M 2\overline{R}^{i\mbox{ }k}_{\mbox{ }i\mbox{ }k} w_{ii}(\partial_{t} w)_{kk} +C(M_{j,m}, V)\delta^3(t)+ C(M_{j,m}, V)\delta^4(t)\\
		&\hspace{4mm}+\frac{1}{2}(||\partial_{t} w||_2^{(\epsilon)})^2(t)+ C(V,\lambda)\delta^2(t).
	\end{align*}
	Integrating from $t=t_0$ to $t=t_0+(k-1)T$:
	\begin{align*}
		||\partial_{t} w||_{L^2([t_0, t_0+(k-1)T];L^2)}+4C||\overline{\nabla}w||(t_0+(k-1)T) &\leq CM_{1,0}^2\delta^2(t_0+(k-1)T)+C(M_{j,m})\delta^4(t)\\
		& \hspace{4mm}+C(V, \lambda)\delta^2(t).
	\end{align*}
	The heat estimates of \cite[Chapter 7, Thm 6]{Evans} imply that
	\begin{align*}
		||w||_2^{(\epsilon)}(t)+||\overline{\nabla}w||_2^{(\epsilon)}(t) &\leq C_0(||F||_{L^2([t_0, t_0+(k-1)T])}^{(\epsilon)}+ ||w||_{H^1}^{(\epsilon)}(t_0+(k-1)T))\\
		&\leq C_0(4VT^{\frac{1}{2}}+M_{1,0}+M_{2,0}))\delta(t). 
	\end{align*}
	for $t\in [t_0+ (k-1)T, t_0+kT]$
	
	Integrating from $t=t_0+(k-1)T$ to $t=t$ gives
	\begin{align*}
		||\partial_{t} w||_{L^2([ t_0+(k-1)T, t];L^2)}+2C||\overline{\nabla}w||(t) &\leq
		4C||\overline{\nabla}w||(t_0+(k-1)T)\\
		&\hspace{4mm}+ \int_{t_0+(k-1)T}^{t}C(M_{1,0}, M_{2,0}, \lambda, V)\delta^2(s) ds\\
		&\leq 4||\overline{\nabla}w||(t_0+(k-1)T)+ TC(M_{1,0}, M_{2,0}, \lambda, V)\delta^2(t)\\
		&\leq CM_{1,0}^2\delta^2(t)+C(M_{j,m})\delta^4(t)+C(V, \lambda)\delta^2(t)\\
		& \hspace{4mm} +TC(M_{1,0}, M_{2,0}, \lambda, V)\delta^2(t),
	\end{align*}
	for $t\in[t_0, t_0+kT].$
	Choose $T$ small enough so that $TC(M_{1,0}, M_{2,0}, \lambda, V)\leq 1$. Then we have the bound on $\overline{w}$.
	
	Let $i\geq 1$, $p\leq 7$. Time-evolution equation for $\overline{\nabla}^i w_{t^p}$ is,
	\begin{align*}
		\partial_t (\overline{\nabla}^i \partial_{t}^p w)&= g^{\alpha\alpha}\overline{\nabla}_{\alpha}\overline{\nabla}_\alpha (\overline{\nabla}^i \partial_{t}^p w)+ \sum_{m+j=i, q\leq p} \overline{\nabla}R * \overline{\nabla}^m \partial_{t}^q w + \overline{\nabla}^{i-1}\partial_t^{p}(\overline{R}^{i\mbox{ }k}_{\mbox{ }i\mbox{ }k} w_{ii})\\
		&\hspace{4mm}+\overline{\nabla}^{i-1}\partial_t^{p}(\delta_{rk}((\overline{b}^{ll}\overline{\nabla}_rw_{ll})^2+\overline{c}^{ll}_r\overline{\nabla}_r w_{ll})+b^{ii}w_{ii}w_{kk}+c^{ii}_{kk}w_{ii}+ \tilde{c}_{kk}w_{kk})+F).
	\end{align*}
	For $t\in [t_0, t_0+(k-1)T]$, we have
	\begin{align*}
		&\left|\left|\partial_t^p\overline{\nabla}^{i-1} F\right|\right|_2^{(\epsilon)} \leq C(V)\delta(t),\\
		&\left|\left|\partial_t^p\overline{\nabla}^{i-1} (\delta_{rk}((\overline{b}^{ll}\overline{\nabla}_rw_{ll})^2+\overline{c}^{ll}_r\overline{\nabla}_r w_{ll})+b^{ii}w_{ii}w_{kk}+c^{ii}_{kk}w_{ii}+ \tilde{c}_{kk}w_{kk}))\right|\right|_2^{(\epsilon)}\leq C(V)\delta^2(t),\\
		&\left(||\sum_{m+j=i, q\leq p} \overline{\nabla}R * \overline{\nabla}^m \partial_{t}^q w=||_2^{(\epsilon)}\right)^2 \leq Q(M_{j,m;j<i, m\leq p})\delta^2(t).
	\end{align*}
	Thus, for $t\in [t_0, t_0+(k-1)T]$, 
	\begin{align*}
		\partial_t(||\overline{\nabla}^{i-1}w_{t^{p}}||_{2}^{(\epsilon)})^2
		&= \int_M \langle \partial_t \overline{\nabla}^{i-1}\partial_{t}^p w, \partial_t \overline{\nabla}^{i-1}\partial_{t}^p w \rangle\\
		&\leq \int_M \langle g^{\alpha\alpha} \overline{\nabla}_\alpha \overline{\nabla}_\alpha \overline{\nabla}^{i-1}\partial_{t}^p w, \partial_t \overline{\nabla}^{i-1}\partial_{t}^p w \rangle+
		\frac{3}{4}(||\overline{\nabla}^{i-1}\partial_{t}^p w||_2^{(\epsilon)})^2\\
		&
		\hspace{4mm}+Q(M_{j,m;j<i,m\leq p})\delta^2(t)+ C(M_j)\delta^4(t)+C(V)\delta^2(t)\\
		&\leq -\partial_t (||\overline{\nabla}^{i}\partial_{t}^p w||_2^{(\epsilon)})^2+ \frac{3}{4}(||\overline{\nabla}^{i-1}\partial_{t}^p w||_2^{(\epsilon)})^2+ Q(M_{j,m;j<i,m\leq p})\delta^2(t)\\
		&\hspace{4mm}+ C(M_{j,m})\delta^4(t)+C(V)\delta^2(t).
	\end{align*}
	Integrating for $t=t_0$ to $t=t_0+(k-1)T$:
	\begin{align*}
		\int_{t_0}^{t_0+(k-1)T}\frac{1}{4}(|| \overline{\nabla}^{i-1}\partial_{t}^p w||_{2}^{(\epsilon)})^2(s)ds + C\partial_t (||\overline{\nabla}^{i}\partial_{t}^p w||_2^{(\epsilon)})^2(t)&\leq  Q(M_{j,m;j<i,m \leq p}, \lambda)\delta^2(t_0+(k-1)T)\\
		&\hspace{4mm}+ C(M_{j,m}, \lambda)\delta^4(t_0+(k-1)T)\\
		&\hspace{4mm}+C(V, \lambda)\delta^2(t_0+(k-1)T).
	\end{align*}
	The heat estimates of \cite[Chapter 7, Thm 6]{Evans} for $t\in [t_0+(k-1)T, t_0+kT]$ gives
	\begin{align*}
		||\overline{\nabla}^i \partial_{t}^p w||_2^{(\epsilon)}(t)+ (||\overline{\nabla}^{i-1} \partial_{t}^p w||_2^{(\epsilon)})^2(t)\leq C(1 + \sum_{j\leq i, m\leq p}M_{j,m})\delta^2(t).
	\end{align*}
	Repeating the estimate of $(||\overline{\nabla}^i \partial_{t}^p w||_2^{(\epsilon)})^2(t)$ for $t\in [t_0+(k-1)T, t_0+kT]$ and integrating from $t=t_0+(k-1)T$ to $t=t$ yields
	\begin{align*}
		\int_{t_0+(k-1)T}^{t}\frac{1}{4}(|| \overline{\nabla}^{i-1}\partial_{t}^p w||_{2}^{(\epsilon)})^2 + \partial_t (||\overline{\nabla}^{i}\partial_{t}^p w||_2^{(\epsilon)})^2(t)&\leq C\partial_t (||\overline{\nabla}^{i}\partial_{t}^p w||_2^{(\epsilon)})^2(t_0+(k-1)T)+\\
		& \hspace{4mm} TC(M_{j,m})\delta^2(t)\\
		&\leq Q(M_{j,m};j<i,m\leq p), \lambda)\delta^2(t)\\
		&\hspace{4mm}+ C(M_{j,m}, \lambda)\delta^4(t)+C(V, \lambda)\delta^2(t)\\
		&\hspace{4mm}+ TC(M_{j,m})\delta^2(t).
	\end{align*}
	Choose $T$ small enough so so that $TC(M_{j,m})\leq 1.$ Then for $N$ negative enough, we have 
	$$\partial_t (||\overline{\nabla}^{i}\partial_{t}^p w||_2^{(\epsilon)})^2(t) \leq Q'(M_{j,m};j<i,m\leq p, V,\lambda)\delta^2(t).$$
	\end{proof}			

\begin{remark}\label{remark}
	From the proof of Theorem \ref{l^2long}, it is clear that there was nothing special about only having $L^2$ bounds of $$\left|\left|\frac{d^j}{dt^j} \overline{\nabla}^iw\right|\right|_2^{\epsilon}(t) \leq M_{i+j}\delta(t),$$ for $i+j\leq 7$. An inductive proof could also be used to prove the same result but for $i+j\leq A$ for any $A\geq 7$. 
\end{remark}

\subsection{Proof of theorem \ref{mainthm}}\label{proofofmain}

\begin{proposition}
	If $\epsilon<\epsilon'$ is sufficiently small, there exists $N<0$ such that the Ricci-de Turck flow 
	\begin{align*}
		&\partial_{t} g(t)= -2\text{Ric}(g(t))+ \mathcal{L}_{V(g(t),g_{0})}g(t), \text{ }g(0)=g_{0},\\
		&g(0)=g_0+ \epsilon h
	\end{align*}
	where $V(g(t),g_{0})$ is a time dependent vector field satisfying
	\begin{equation*}
		g(V(g(t),g_{0})(t), \cdot)= -\text{tr}_{g(t)}\nabla^{g(t)} g_{0}-\frac{1}{2}\nabla^{g(t)} \text{tr}_{g(t)}g_{0},
	\end{equation*}
	with $g(t_0)=g_0+\epsilon h$ exists for $[t_0, N]$, where $\epsilon=e^{\lambda t_0}$.
\end{proposition}

\begin{proof}
	Let $N$ be negative enough. Take $t_0$ negative enough so that $|Rm(g_0+\epsilon h)|_{g_0+\epsilon h} \leq \frac{K}{2}$ for some $K>0$. Then, by Theorem \ref{shorttimeests}, there exists $T=T(K)>0$ such that 
	the above equation has a solution $g(t)$ for $t\in[t_0, t_0+T]$. Suppose $g(t)$ exists on $[t_0, t_0 +(k-1)T]$. 
	Thus, for $\epsilon\leq \epsilon'$ there is a $C(\epsilon')>0$ such that 
	$$||w||_{C^{3}}^{(0)} \leq C \delta(t),$$
	for all $t\in[t_0, t_0+kT]$.
	The Sobelev inequality of Lemma \ref{sobelevineq} applied to the conclusion of Theorem \ref{l^2long} yields
	\begin{equation}\label{C^3}
		||w||_{C^{3}}^{(\epsilon)} \leq C(M_{j,m}) \delta(t),
	\end{equation}
	for all $t\in[t_0, t_0+(k-1)T]$.
	
	Then, by the equivalence of $\overline{g}$ and $g(t)$,  we have $|| w||^\epsilon_{C^3} \leq C\delta(t)$. If $N$ is negative enough to make $C\delta(t)$ is small enough so that 
	$$|Rm(g(t_0+(k-1)T))|_{g(t_0+(k-1)T)} \leq K,$$
	then, by Theorem \ref{shorttimeests}, $g(t)$ can be extended to a solution on $t\in[t_0, t_0+kT]$.
\end{proof}

\begin{theorem}\label{anicentRDTF}
	There exists a non-trivial ancient solution to the Ricci-de Turck flow coming out of $g_0$.
\end{theorem}			

\begin{proof}
	Pick a sequence $\epsilon_n \rightarrow 0$. Then, by the above Proposition, there exists a family of metrics $g^n(t)$ such that
	\begin{align*}
		&\partial_{t} g^n(t)= -2\text{Ric}(g^n(t))+ \mathcal{L}_{V(g^n(t),g_{0})}g^n(t), 
	\end{align*}
	where $V(g^n(t),g_{0})$ is a time dependent vector field satisfying
	\begin{equation*}
		g(V(g^n(t),g_{0})(t), \cdot)= -\text{tr}_{g^n(t)}\nabla^{g^n(t)} g_{0}-\frac{1}{2}\nabla^{g^n(t)} \text{tr}_{g^n(t)}g_{0},
	\end{equation*}
	with $g^n(t_n)=g_0+\epsilon_n h$ exists on $[t_n, N]$, where $\epsilon_n=e^{\lambda t_n}$.
	Let $M_i\subset M$ be a sequence of compact subsets such that $\cup_i M_i=M$. Set $P_i=D_i\times [t_0, N].$ By Theorem \ref{l^2long}, Lemma \ref{sobelevineq}, and the equivalence of $g_0$ and $g_0+\epsilon_n h$, 
	$$\left|\frac{d^j}{dt^j}\overline{\nabla}^i g^n \right|_0,$$
	for $i+j\leq 4$ is bounded uniformly on $P_n$ for all $n$. By Arzelà–Ascoli theorem and a diagonal argument, we can assume that
	$$g^n, \overline{\nabla} g^n, \overline{\nabla}^2 g^n, \partial_t g^n$$
	converge uniformly on $P_i$. Thus, the limit $$g^n \rightarrow g$$
	with respect to $|\cdot |_0$
	is a solution to the Ricci-de Turck flow on $M\times (-\infty, N].$ 
	
	Let $k\geq 0$.	We claim that $g(t)\rightarrow g_0$ in $C^k$ as $t\rightarrow -\infty$. Let $\epsilon>0$ be small. Choose $\overline{t}<N$ such that $\delta(\overline{t})\leq \epsilon.$ Let $t\leq \overline{t}$. Let $n$ be such that $t_n \leq t$ and, by Remark \ref{remark} and Lemma \ref{sobelevineq}, such that 
	$$||g-g^n(t)||^{(0)}_{C^k} \leq \epsilon.$$
	Then, by (\ref{C^3}), 
	\begin{align*}
		||g(t)-g_0||^{(0)}_{C^k} &\leq ||g(t)-g^n(t)||^{(0)}_{C^k}+||g_n(t)-g^n(t_n)||^{(0)}_{C^k}+||g^n(t_n)-g_0||^{(0)}_{C^k}\\
		& \leq \epsilon+ ||w^n(t)-(\delta(t)-\delta(t_n))h||^{(0)}_{C^k}+\delta(t_n)||h||_{C^k}^{(0)}\\
		&\leq C(\delta(t)+\epsilon)\\
		&\leq 2C\epsilon
	\end{align*}
	for $C$ independent of $n$. Therefore, $g(t)\rightarrow g_0$ in $C^k$ as $t\rightarrow -\infty$.

	We still have to show that $g(t)$ is not the trivial solution $g(t)=g_0$.
	By equation (\ref{weq}), $w^n= g^n(t)-g_0-\delta(t)h$ satisfies
	$$\partial_t w^n= L(w^n)+O(\delta^2(t)),$$
	for $L$ an elliptic linear operator with coefficients uniformly bounded in $n$. Since $w^n(t_n)=0$, parametric method of \cite{Friedman} implies that
	$$\left|w^n \right|_\epsilon\leq C\delta^2(t),$$
	for all $t\in [t_n, N]$ and $C>0$ independent of $n$. Hence,
	$$g^n(t)=g_0+\delta(t)h+O(\delta^2(t)),$$ and so 
	\begin{equation}\label{nondiffeo}
	g(t)=g_0+\delta(t)h+O(\delta^2(t)).
	\end{equation} 
	Therefore, $g(t)\not= g_0$ for all $t$.
\end{proof}

\begin{proof} (of Theorem \ref{mainthm})
	Theorem \ref{anicentRDTF} gives us a non-trivial ancient solution $g(t)$ to the Ricci-de Turck flow coming out of $g_0$. Since $|Rm(g(t))|$ is bounded uniformly on $M\times (-\infty, N]$, the de Turck diffeomorphisms can be solved on $M\times (-\infty, N]$. Pulling back by these diffeomorphisms yields a ancient solution to the Ricci flow such that for all $k\geq 0$, we have $g(t)\rightarrow g_0$ (modulo diffeomorphisms) in $C^k$ as $t \rightarrow -\infty$. This Ricci flow solution is non-trivial since $h$ is traceless and divergence-free and so, by (\ref{nondiffeo}), $g(t)$ is not just a diffeomorphism and rescaling of $g_0$. 
\end{proof}

	\bibliography{refs}
\bibliographystyle{amsplain}

\end{document}